\documentclass[10pt]{amsart}
\usepackage[utf8]{inputenc}
\usepackage{amsfonts,amssymb,amsmath,amsthm}
\usepackage[a4paper, total={6in,9in}]{geometry}

\usepackage[utf8]{inputenc}
\usepackage{graphicx}
\usepackage{mathtools}

\usepackage{subcaption}
\usepackage{tikz}
\usepackage{enumerate}
\usepackage{cases}
\usepackage{multicol}
\usepackage{sidecap}
\usepackage{float}
\usepackage{xcolor}
\definecolor{blue2}{rgb}{0.67, 0.9, 0.93}
\usepackage[color=blue2]{todonotes}
\usepackage{cases}
\usepackage{extarrows}
\usepackage{soul}

\usepackage[normalem]{ulem}
\usepackage{hyperref}
\usepackage{setspace}
\numberwithin{equation}{section}
\newtheorem{theorem}{Theorem}[section]
\newtheorem{lemma}[theorem]{Lemma}

\theoremstyle{definition}
\newtheorem{definition}[theorem]{Definition}

\newtheorem{example}[theorem]{Example}
\newtheorem{corollary}[theorem]{Corollary}
\newtheorem{remark}[theorem]{Remark}

\newcommand{\N}{{\mathbb N}}
\newcommand{\R}{{\mathbb R}}

\newcommand{\Z}{\mathbb{Z}}

\newcommand{\T}{{\mathbb{T}}}

\usepackage{xcolor}
\definecolor{refkey}{gray}{.75}
\definecolor{labelkey}{gray}{.75}

\newcommand{\wc}{\ensuremath{w^\mathcal{C}}}

\usepackage[pagewise]{lineno}

\begin{document}
\author{Daniel \v{S}pale}      %
        \address{Department of Mathematics and NTIS, Faculty of Applied Sciences, University of West Bohemia, Univerzitn\'{\i} 8, 306 14 Plze\v n, Czech Republic}
       \email{dspale@kma.zcu.cz}
\author{Petr Stehl\'{i}k}      %
        \address{Department of Mathematics and NTIS, Faculty of Applied Sciences, University of West Bohemia, Univerzitn\'{\i} 8, 306 14 Plze\v n, Czech Republic}
       \email{pstehlik@kma.zcu.cz (corresponding author)}

\title{Traveling waves in Bistable Reaction-Diffusion Cellular Automata}

\begin{abstract}
We describe various types of traveling fronts of bistable reaction-diffusion cellular automata. These dynamical systems with discrete time, space, and state spaces can be seen as fully discrete versions of widely studied bistable reaction-diffusion equations. We show that moving traveling waves for high diffusion parameters are restricted to slow speeds and their profiles are interestingly not unique. Pinned waves always exist for weak diffusion as in the case of lattice equations but do not complement parametric region of moving traveling waves. The remaining parameter domain is dominated by waves which are unique to cellular automaton settings. These higher-order traveling waves move and periodically change profile at the same time.
\end{abstract}



        \date{\today}
        



\maketitle

\section{Introduction}\label{s:introduction}
Waves in lattice reaction-diffusion equations provide one of the most fascinating distinctive behavior between the dynamics of spatiotemporal dynamical systems which differ only in underlying discrete and continuous structures. As the most typical example, the Nagumo partial differential equation (i.e., continuous space) serves as a textbook example for the phase plane analysis which provides existence of traveling waves \cite{Logan1994}. In contrast, the Nagumo lattice differential equation (discrete space) is far more complicated to study since it yields several new dynamical phenomena, e.g., pinning of waves \cite{Keener1987}, spatial topological chaos \cite{Chow1998}, and existence of nonmonotone waves \cite{Hupkes2019c}.

In this paper, we study traveling fronts of a class of bistable reaction-diffusion cellular automata, which correspond to Nagumo partial and lattice differential equations but are fully discrete, i.e., temporal, spatial, and state sets are all discrete. In Section~\ref{s:preliminaries}, we introduce a bistable discrete dynamical system
\begin{equation}\label{e:ca:general}
u(t+1)=F(u(t)),\quad t\in\N_0,\ u(t)\in S^\Z,    
\end{equation}
with  the finite state space $S\subset \N_0$ and the mapping $F:S^\Z\rightarrow S^\Z$ which involves both diffusion and reaction processes. We then compare its behavior to the widely studied corresponding partial and lattice differential equations and study three types of fronts - moving, pinned, and periodically changing higher-order traveling waves.

Beyond the natural interest in distinct mathematical properties, our motivation is driven by several applications where discrete time, space, and states are common, e.g., seasonal populations (discrete time, \cite{Allen2007}), island or pond populations (discrete space, \cite{Hanski2002}), or very small populations (discrete states, \cite{Vries2006}).

\subsubsection*{(Nonspatial) bistable dynamics.} Various biological, chemical and mechanical systems possess two stable equilibria. The simplest models are connected to ordinary differential models with cubic nonlinearities, e.g.,
\[
u'(t)=u(t)\left(1-u^2(t)\right),\quad t\geq 0,
\]
with stable equilibria $u^*_{1,2}=\pm 1$ and the unstable equilibrium $u_3^*=0$. In mathematical biology, it is more natural to generalize the standard logistic growth with a capacity $K>0$ into the bistable models with the so called Allee effect. Roughly speaking, the (demographic) Allee effect captures phenomena in which populations smaller than a viability threshold $a$ with $0<a<K$ go extinct. Mathematically, this leads to the bistable model
\begin{equation}\label{e:Allee}
u'(t)=\lambda u(t)\left(\frac{u(t)}{A}-1\right)\left(1-\frac{u(t)}{K}\right),\quad t\geq 0,
\end{equation}
with the stable equilibria $u^*_{1}=0$ (extinction), $u^*_{2}=K$ (carrying capacity)  and unstable equilibrium $u_3^*=A$ (viability threshold). Numerous examples from applications, list of alternative mathematical approaches, and deep discussions about the Allee effect and nonspatial bistable dynamics can be found in \cite{Courchamp2008}. 



\subsubsection*{Waves in bistable partial differential equations.} Nondimensionalisation of \eqref{e:Allee} and the use of the standard continuous spatial diffusion leads to the widely-studied Nagumo bistable partial differential equation (PDE)
\begin{equation}\label{e:NagumoPDE}
u_t(x,t)=du_{xx}(x,t) + u(x,t)\left(u(x,t)-a\right)\left(1-u(x,t)\right),\quad t\geq 0,\quad x\in\mathbb{R},
\end{equation}
where $d>0$ denotes the diffusion parameter. Initially, this equation has been considered as a simplified model for the pulse propagation  in nerve axons \cite{Nagumo1962}. The PDE has a traveling wave solution which connects the two stable equilibria $u^*_{1}=0$ and $u^*_{2}=1$  \cite{Logan1994}
\begin{equation}\label{e:NagumoPDE:sol}
u(x,t) = \frac{1}{2} + \frac{1}{2}\tanh\left(\frac{x-ct}{2\sqrt{2d}}\right),    
\end{equation}
where $c$ denotes the speed of propagation and satisfies
\[
c=\sqrt{2d} \left(a-\frac{1}{2} \right),
\]
i.e., the wave profile \eqref{e:NagumoPDE:sol} propagates left whenever $a<1/2$, right for $a>1/2$, and is pinned when $a=1/2$, see the left panel of Fig.~\ref{fig:speed:comparison}. 

\begin{figure}
    \centering
    \includegraphics[width=.34\textwidth]{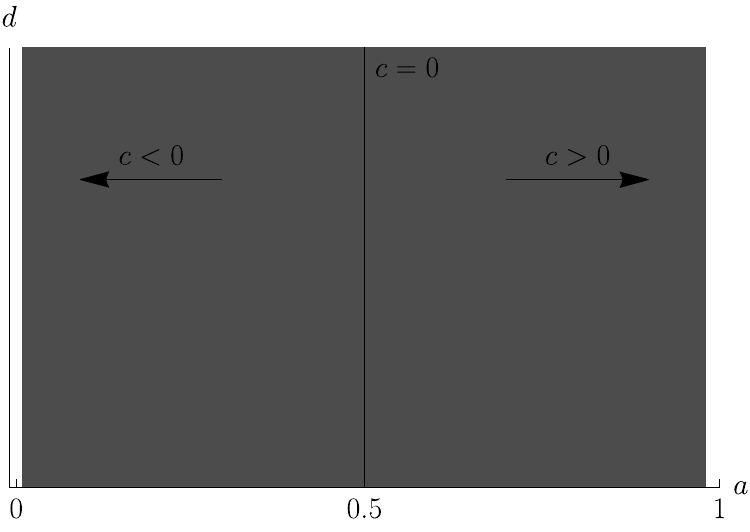}
    \includegraphics[width=.34\textwidth]{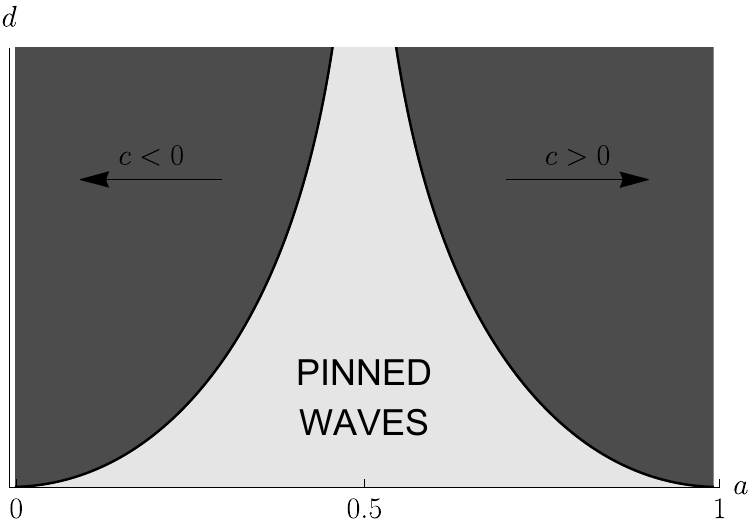}
    \includegraphics[width=.24\textwidth]{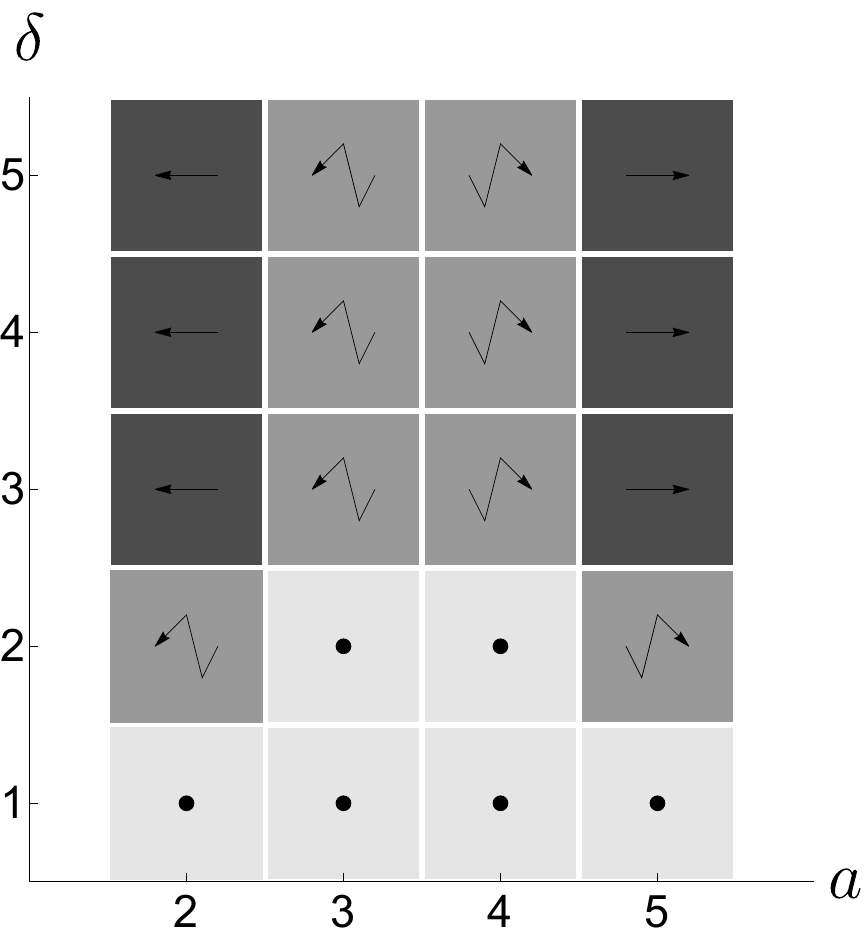}
    \caption{Comparison of properties of traveling waves in $(a,d)$-plane for bistable partial differential equation (left panel), lattice differential equation (central panel) and cellular automaton (right panel). Dark regions represent left or right moving traveling waves, light regions pinned waves and middle gray regions correspond to higher-order waves which are unique to cellular automata.}
    \label{fig:speed:comparison}
\end{figure}

\subsubsection*{Waves in bistable lattice  differential equations.} In many applications (sequence of cells, neurons, spatial habitats) the continuous space $X=\mathbb{R}$ in \eqref{e:NagumoPDE} is naturally replaced by an integer lattice $X=\mathbb{Z}$ and we obtain the Nagumo lattice equation
\begin{equation}\label{e:NagumoLDE}
u'_i(t)=d\left(u_{i-1}(t)-2u_i(t)+u_{i+1}(t) \right) + u_i(t)\left(u_i(t)-a\right)\left(1-u_i(t)\right),\quad t\geq 0,\quad i\in\mathbb{Z}.
\end{equation}
The same equation can be obtained from the PDE \eqref{e:NagumoPDE} by applying standard central difference scheme to the second spatial derivative.

The dynamical properties of the LDE \eqref{e:NagumoLDE} differ significantly from those of the PDE \eqref{e:NagumoPDE}. There is no explicit solution like \eqref{e:NagumoPDE:sol}. For any $a\in(0,1)$ and sufficiently small diffusion $d$ the waves are pinned with $c=0$ and there exists large number of stationary solutions \cite{Keener1987}. For any $d>0$ there exists $\varepsilon>0$ such that $c=0$ whenever $\left\lvert a-1/2\right\rvert<\varepsilon$, see \cite{Chow1998} and the central panel in Fig.~\ref{fig:speed:comparison}. The exact boundaries between $c=0$ and $c\neq 0$ remain unknown and the details of the pinning procedures have been intensively studied in simplified bistable caricatures \cite{Fath1998}.

Alternative propagation schemes have been observed in the pinning regions, e.g., nonmonotone multichromatic waves \cite{Hupkes2019c}. Similar pinning behavior has been studied as well in the models when also time is considered to be discrete (or discretized by the forward Euler discretization scheme), e.g., \cite{Chow1995},

\begin{equation}\label{e:NagumoCML}
u_i(t+1)-u_i(t)=d\left(u_{i-1}(t)-2u_i(t)+u_{i+1}(t) \right) + u_i(t)\left(u_i(t)-a\right)\left(1-u_i(t)\right),\quad t\in \mathbb{N}_0, i\in\mathbb{Z}.
\end{equation}

\subsubsection*{Patterns and waves in bistable cellular automata.} There have been numerous attempts to study patterns and waves in bistable cellular automata, i.e., models in which not only time and space are discrete but also the values of $u$ are reduced to a finite number of states. In many cases, the focus has been on the discretization techniques, standard finite difference schemes \cite{Weimar1997}, the tropical discretization \cite{Murata2013}, or the fine correspondence of three-letter state set model with the behavior of lattice equations \eqref{e:NagumoLDE}--\eqref{e:NagumoCML}, \cite{Chow1996}. The natural presence of propagation phenomena has been identified in reaction-diffusion cellular automata by techniques ranging from simulations \cite{Deutsch2017} to a systematic combinatorial  analysis \cite{Courbage1999}. Numerous wave types include periodic waves \cite{Adamatzky2013, Courbage1997, Schonfisch1995} and spiral waves \cite{Deutsch2017} based on the famous excitable cellular automata by Greenberg and Hastings \cite{Greenberg1978}. Large number \cite{Spale2022} and numerous types \cite{Chow1996} of stationary patterns of reaction-diffusion automata have also been observed in reaction-diffusion automata.

\subsubsection*{Higher-order waves.} In this paper we focus on a class of bistable cellular automata and compare their propagation properties to those of the PDE \eqref{e:NagumoPDE} and the LDE \eqref{e:NagumoLDE}. We show that beyond monotonic moving traveling waves and pinned waves there exists an intermediate region in which there are no moving or pinned traveling waves. The dynamics in this region are driven by a new type of waves which are unique to cellular automata, the \emph{higher order traveling waves}. These are the systems of periodically changing and yet moving profiles. More precisely, a  higher-order $(c,m)$-traveling wave is a set of $m$ profiles which periodically repeat themselves and each periodic occurrence is coupled with a shift/speed $c$, see the right panel of Fig.~\ref{fig:example_higher_order_waves} for illustration.

As far as we know, these higher-order waves connecting stationary solutions have not been described so far. However, periodically repeating and moving patterns (not waves though) are known to exist for cellular automata \cite{Wolfram2002}. For example, there are elementary cellular automata leading to moving patterns, e.g., the rule 2, see Fig.~\ref{fig:rule2} and there are also numerous which lead to the periodically repeating and moving patterns, e.g., the rule 3, see Fig.~\ref{fig:rule3}.


\begin{figure}
    \centering
    \includegraphics[width=.75\textwidth]{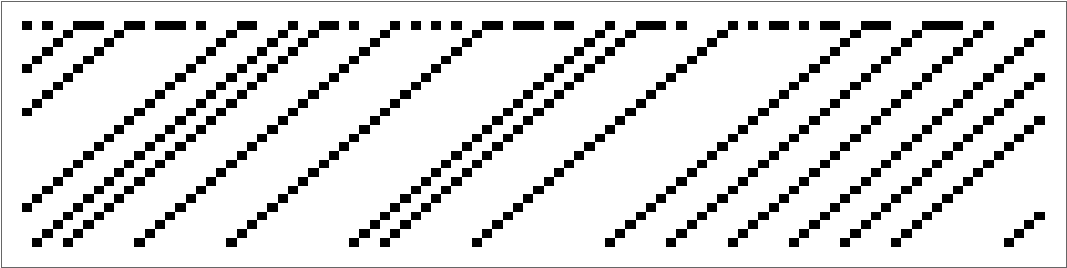}
    \caption{Elementary cellular automaton, rule 2, whose dynamics lead to  moving patterns. The standard visualization of elementary cellular automata is used, i.e., time axis runs from the top, see \cite{Deutsch2017, Wolfram2002}.}
    \label{fig:rule2}
\end{figure}

\begin{figure}
    \centering
    \includegraphics[width=.75\textwidth]{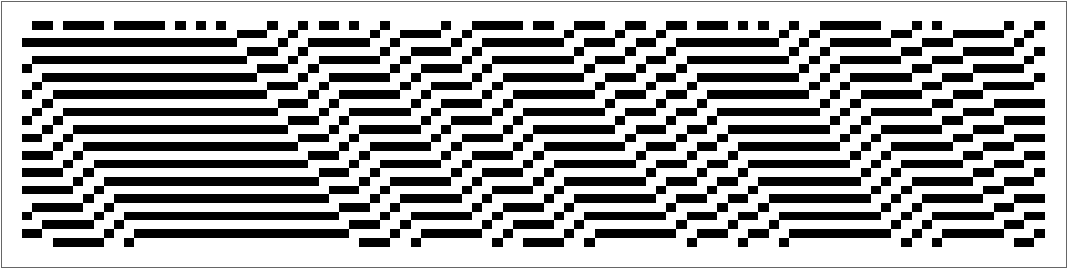}
    \caption{Elementary cellular automaton, rule 3, whose dynamics converge to periodically repeating and moving patterns, two interchanging profiles shift by 1 each second period. This behavior is similar to higher-order waves which we cover in this manuscript.}
    \label{fig:rule3}
\end{figure}







\subsubsection*{Paper Structure.} In Section~\ref{s:preliminaries} we properly define the details of the model of the bistable cellular automaton~\eqref{e:ca:general}. In Section~\ref{s:traveling} we study moving traveling waves and provide a necessary and sufficient condition for their existence in Thm.~\ref{t:ns:left}. In Section~\ref{s:pinned} we then analyze pinned waves and show that there exist parameters for which neither moving nor pinned waves exist. In Section~\ref{s:higher} we show that these parameters yield higher-order traveling waves, profiles which both move and periodically change. We conclude, in Section~\ref{s:examples:TW}, by simulation examples that illustrate existence of moving, pinned, and higher-order waves and their dependence on various parameters.


\section{Reaction-diffusion cellular automata}\label{s:preliminaries}
In this section we briefly introduce basic concepts of bistable reaction-diffusion cellular automata of type \eqref{e:ca:general}. For more details and discussion on alternative approaches, see our companion paper \cite{Spale2022} which deals with stationary patterns.
\subsubsection*{Discrete-state diffusion.}
We define the symmetric discrete-state diffusion $d_{\delta}:(\N_0)^3 \rightarrow \N_0$ with a diffusion parameter $\delta\in\mathbb{N}$ by
\begin{equation}\label{e:d_delta}
    d_{\delta}(m,n,p) = n +\underbrace{h_{\delta}(m,n)}_{\text{left-side dynamics}} + \underbrace{h_{\delta}(p,n)}_{\text{right-side dynamics}},
\end{equation}
where the left- and right-side dynamics are determined by an auxiliary function $h_{\delta}:(\N_0)^2 \rightarrow \N_0$ 
\begin{equation}\label{e:h_delta}
    h_{\delta}(m,n) := \begin{cases}
    \delta, & m-n\geq 2\delta,  \\
    -\delta, & m-n\leq -2\delta, \\
    \delta-1, & m-n\in \{2(\delta-1),2\delta-1\},  \\
    -(\delta-1), & n-m\in \{2(\delta-1),2\delta-1\}, \\
    \vdots \\
    1, & m-n \in \{2,3\},  \\
    -1, & n-m\in \{2,3\}, \\
    0, & \text{otherwise}.  \\
	\end{cases}
\end{equation} 
We then introduce a diffusion automaton $D: \N_0^{\Z} \rightarrow \N_0^{\Z} $ via
\begin{equation}\label{e:diffusion:automaton}
    D(u) = 
    \begin{pmatrix}
    \vdots \\
    d_\delta(u_{x-2},u_{x-1},u_{x}) \\
    d_\delta(u_{x-1},u_{x},u_{x+1})\\
    d_\delta(u_{x},u_{x+1},u_{x+2})\\
    \vdots
    \end{pmatrix},\quad u \in \N_0^{\Z}.
\end{equation}

\subsubsection*{Discrete-state bistable reaction function.} Focusing on a bistable, discrete-state and local (i.e., nonspatial) dynamics $v(t+1)=f(v(t))$, $v(t)\in\N_0$, we introduce bistable maps $f$ with the help of a capacity $K\in\N$ and a viability constant $a\in\N$ satisfying $0<a<K$, see also Fig.~\ref{fig:reaction_waveprofile_comparison} and cf.~\eqref{e:Allee}.  
A function $f:\N_0\rightarrow\N_0$ is said to be a (discrete) bistable reaction function if it is nondecreasing and there exist $a,K\in\N$, $2\leq a \leq K-2$ such that $f(0)=0$, $f(a)=a$, $f(K)=K$, and
\begin{equation}\label{e:bistable}
f(u)\in\begin{cases}
    [0,u)_\Z &\text{if } u\in (0,a)_\Z, \\
    (u,K]_\Z &\text{if } u\in (a,K)_\Z, \\
    [K,u)_\Z &\text{if } u\in (K,\infty)_\Z, \\
    \end{cases}
\end{equation}
where the subscript $\Z$ denotes discrete intervals, e.g., $(0,a)_\Z=(0,a)\cap \Z$. See the left panel of Figure~\ref{fig:reaction_waveprofile_comparison}.


\subsubsection*{Reaction-diffusion cellular automata.} Combining the spatial and local  dynamics (i.e., dif\-fu\-sion~\eqref{e:diffusion:automaton} and reaction~\eqref{e:bistable}), we construct bistable reaction-diffusion automata~\eqref{e:ca:general}. Consider a time scale $\T=\N_0$, one-dimensional discrete lattice $X=\Z$, a state set $S\subset \N_0$ and an update rule $F:S^X\rightarrow S^X$ given by

\begin{equation}\label{e:RDCA:F}
    F(u) = 
    \begin{pmatrix}
    \vdots \\
    d_\delta(f(u_{x-2}),f(u_{x-1}),f(u_{x})) \\
    d_\delta(f(u_{x-1}),f(u_{x}),f(u_{x+1}))\\
    d_\delta(f(u_{x}),f(u_{x+1}),f(u_{x+2}))\\
    \vdots
    \end{pmatrix},\quad u \in S^{\Z},
\end{equation}
with $d_\delta:S^3\rightarrow S$ given by \eqref{e:d_delta} and $f:S\rightarrow S$ being a bistable map (see~\eqref{e:bistable}). We then call the quadruple $\mathcal{C}=(\T,X,S,F)$ the bistable reaction-diffusion cellular automaton (abbreviated by RDCA) and study the induced discrete dynamical system
\begin{equation}\label{e:RDCA:DDS}
u(t+1)=F(u(t)), \quad t\in\T,\ u\in S^\Z.
\end{equation}
We construct the reaction-diffusion automaton via the composition of the reaction and diffusion (as opposed to summing them up as in continuous-state models~\eqref{e:NagumoPDE}, \eqref{e:NagumoLDE}, see the interesting discussion in \cite{Pospisil2020} on this topic). Our motivation stems, among other properties, from the fact that such an approach ensures the interval invariance, see \cite[Rem.~2.3]{Spale2022} for more details.

\section{Moving traveling waves}\label{s:traveling}

In this paper we focus on traveling wave solutions of bistable RDCA~\eqref{e:RDCA:DDS}, i.e., solutions which can be written as
\begin{equation}\label{e:wave:ansatz}
u_n(t)=w_{n-c\cdot t},\quad  t\in \N_0,\quad n\in\Z,
\end{equation}
for some wavespeed $c\in\N$ and an increasing wave profile
\begin{equation}\label{e:wave}
w=(\ldots,w_{-2},w_{-1},w_0,w_1,w_2,\ldots,w_{N},w_{N+1},\ldots)
\end{equation}
with
\begin{equation}\label{e:wave:core}
\ldots=w_{-1}=0=w_0<w_1\leq w_2 \leq w_3 \leq \ldots \leq w_{N-1} < w_{N}=K=w_{N+1}=\ldots,
\end{equation}
for a finite $N \in \N$. We refer to the vector $\wc=(w_1,\ldots,w_{N-1})$ as a \emph{wave core}  and call the value $N-1$ \emph{wave core length}. 

From our Ansatz \eqref{e:wave:ansatz} we derive that the wave profile satisfies
\begin{equation}\label{e:wave:F}
(\dots, w_{-1-c}, w_{-c}, w_{1-c}, \dots) = F\left((\dots, w_{-1}, w_{0}, w_{1}, \dots)\right), \quad c \in \Z.
\end{equation}

We call the profile $w$ given by \eqref{e:wave}--\eqref{e:wave:core} satisfying \eqref{e:wave:F} with $c \in \Z$ a \textit{left traveling wave} for $c<0$, \textit{right traveling wave} for $c>0$, and \textit{pinned traveling wave} for $c=0$. Left or right traveling waves are denoted as \textit{moving traveling waves}, see Fig.~\ref{fig:reaction_waveprofile_comparison}.

\begin{figure}
    \centering
    \includegraphics[width=.45\textwidth]{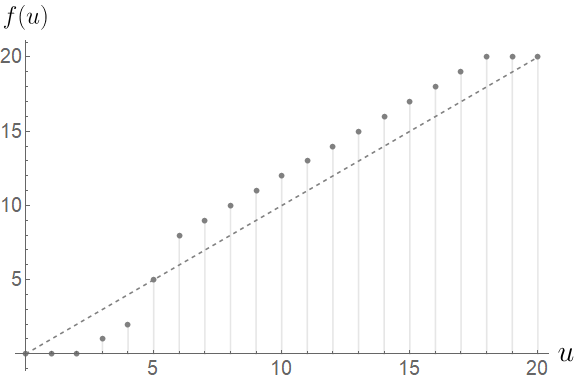}
    \includegraphics[width=.45\textwidth]{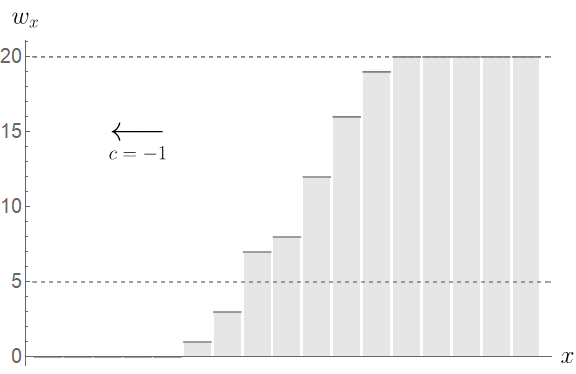}
    \caption{Examples of a bistable reaction function~\eqref{e:bistable} (left panel) and a wave profile~\eqref{e:wave}--\eqref{e:wave:core} (right) with $a=5$ and $K=20$. Throughout the paper we use scatterplots (left) to plot reactions $f$ with the dashed line representing $f(u)=u$. Similarly, we use bar graphs (right) to plot wave profiles with dashed lines illustrating values of $a$ and $K$.}
    \label{fig:reaction_waveprofile_comparison}
\end{figure}

Employing \eqref{e:RDCA:F} we see that the traveling wave satisfies
\begin{equation}\label{e:recurrence:general}
w_{n-c}=f(w_n) + h_\delta(f(w_{n-1}),f(w_n)) + h_\delta(f(w_{n+1}),f(w_n)),\quad n\in\Z.
\end{equation}

It is straightforward to show that $|c|\leq 1$, i.e., there are no `fast' traveling waves with $|c|\geq 2$.

\begin{lemma}[Nonexistence of traveling waves for $|c|\geq 2$]\label{l:nonexistence:fast}
Let $\delta, a, K \in \N$ be such that $2\leq a \leq K-2$ and the bistable reaction function $f$ be defined by \eqref{e:bistable}. Then there are no traveling waves \eqref{e:wave}--\eqref{e:wave:core} in  the bistable RDCA~\eqref{e:RDCA:DDS} with $|c|\geq2$.
\end{lemma}
\begin{proof}
Assume by contradiction that there exists a left traveling wave with $c=-2$, i.e., 
\[
(\dots, w_{1}, w_{2}, w_{3}, \dots) = F\left((\dots, w_{-1}, w_{0}, w_{1}, \dots)\right).
\]
Equations \eqref{e:d_delta} and \eqref{e:RDCA:F} then imply the system of following equations
\[
w_{n+2}=f(w_n) + h_\delta(f(w_{n-1}),f(w_n)) + h_\delta(f(w_{n+1}),f(w_n)), \quad n \in \Z.    
\]
If $(w_{n-1}, w_{n}, w_{n+1})=(0,0,0)$, then $w_{n+2} = 0$ and we can repeat the argument for all $n \in \Z$ to get that $w\equiv 0$, a contradiction with \eqref{e:wave:core}.

Similar arguments can be used in order to prove the same for all $|c|\geq2$.
\end{proof}

Straightforwardly, we can only focus on left traveling waves with negative speed $c<0$, the results for right-traveling waves with positive speed $c>0$ translate automatically.

\begin{lemma}[Existence of right traveling waves with increasing profiles]\label{l:right:traveling}
Let $\delta, a, K \in \N$ be such that $2\leq a \leq K-2$ and the bistable reaction function $f$ be defined by \eqref{e:bistable}. A wave profile $w$ \eqref{e:wave}--\eqref{e:wave:core} satisfies \eqref{e:wave:F} with $c =-1$ if and only if the configuration $\widehat{w}$  defined by $\widehat{w}_n=K-w_{N-n}$
is a right traveling wave of the form \eqref{e:wave}--\eqref{e:wave:core} of RDCA \eqref{e:RDCA:DDS} with $\widehat{f}(u)=K-f(K-u)$ for all $u\in(0,K)_\Z$, i.e., satisfies \eqref{e:wave:F} with $\widehat{c}=1$.
\end{lemma}
\begin{proof}
    First, let us observe that $\widehat{f}$ is a bistable reaction function~\eqref{e:bistable} with $\widehat{a}=K-a$.

    If $w$ is a left traveling with $c=-1$, then the equation \eqref{e:wave:F} implies
    \begin{equation}\label{e:right:traveling:pr}
        w_{n+1}=f(w_n) + h_\delta(f(w_{n-1}),f(w_n)) + h_\delta(f(w_{n+1}),f(w_n)), \quad n \in \Z.
    \end{equation}
    Applying $\widehat{f}(u)=K-f(K-u)$ and $\widehat{w}_n=K-w_{N-n}$ we get for all $n\in\Z$
    \begin{align*}
        \widehat{w}_{n-1}&=K-w_{N-n+1} \\
        &=K-f(w_{N-n})-h_\delta(f(w_{N-n-1}), f(w_{N-n}))-h_\delta(f(w_{N-n+1}), f(w_{N-n})) \\
        &=K-(K-\widehat{f}(\widehat{w}_n))-h_\delta(K-\widehat{f}(\widehat{w}_{n+1}), K-\widehat{f}(\widehat{w}_{n}))-h_\delta(K-\widehat{f}(\widehat{w}_{n-1}), K-\widehat{f}(\widehat{w}_{n}))\\
        &=\widehat{f}(\widehat{w}_n)+h_\delta(\widehat{f}(\widehat{w}_{n+1}), \widehat{f}(\widehat{w}_{n}))+h_\delta(\widehat{f}(\widehat{w}_{n-1}), \widehat{f}(\widehat{w}_{n})),
    \end{align*}
    which shows that $\widehat{w}$ satisfies \eqref{e:wave}--\eqref{e:wave:core} and \eqref{e:wave:F} with $c =1$. The reverse implication follows by the same computation.
\end{proof}

Similar arguments yield also existence of solutions with decreasing profiles

\begin{lemma}[Existence of right traveling waves with decreasing profiles]\label{l:right:traveling:dec}
Let $\delta, a, K \in \N$ be such that $2\leq a \leq K-2$ and the bistable reaction function $f$ be defined by \eqref{e:bistable}. A wave profile $w$ \eqref{e:wave}--\eqref{e:wave:core} satisfies \eqref{e:wave:F} with $c =-1$ if and only if the reverse configuration $\tilde{w}$ 
\begin{equation}\label{e:wave:reverse}
\tilde{w}=(\ldots,w_{N+1},w_{N},\ldots,w_{2},w_{1},w_0,w_{-1},w_{-2},\ldots),
\end{equation}
is a right traveling wave, i.e., satisfies \eqref{e:wave:F} with $\tilde{c}=1$.
\end{lemma}
\begin{proof}
From~\eqref{e:right:traveling:pr} the reverse configuration $\tilde{w}$ immediately satisfies
\begin{align*}
\tilde{w}_{N-n-1}&= w_{n+1}\\
&=f(w_n) + h_\delta(f(w_{n-1}),f(w_n)) + h_\delta(f(w_{n+1}),f(w_n))\\
&=f(\tilde{w}_{N-n}) + h_\delta(f(\tilde{w}_{N-n+1}),f(\tilde{w}_{N-n})) + h_\delta(f(\tilde{w}_{N-n+1}),f(\tilde{w}_{N-n})).\\
\end{align*}
The reverse implication can be shown equivalently.
\end{proof}

Radial waves thus exist naturally in the cellular automaton settings.

\begin{example}[Radial waves]\label{x:radial}
If $w$ given by \eqref{e:wave}--\eqref{e:wave:core} is a left traveling wave with $c=-1$ then it follows immediately from Lem.~\ref{l:right:traveling:dec} that the profile composed of the `left increasing' and `right decreasing' profile
\[w=(\dots,0,0=w_0,w_1,w_2,\dots,w_{N-1},w_N=K,K,\dots,K,w_{N-1},\dots,w_2,w_1,0,0,\dots)\]
will propagate both to the left and right side as radial waves, see the left panel of Fig.~\ref{fig:ex_radial_wave}. Similarly, we can construct radial wave from `left decreasing' and `right increasing' profiles, see the right panel in Fig.~\ref{fig:ex_radial_wave}. Their existence follows from Lemmas~\ref{l:right:traveling} and~\ref{l:right:traveling:dec}.

\begin{figure}
    \centering
    \includegraphics[width=.45\textwidth]{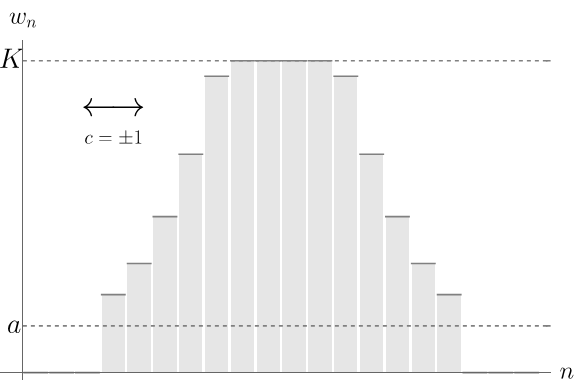}
    \includegraphics[width=.45\textwidth]{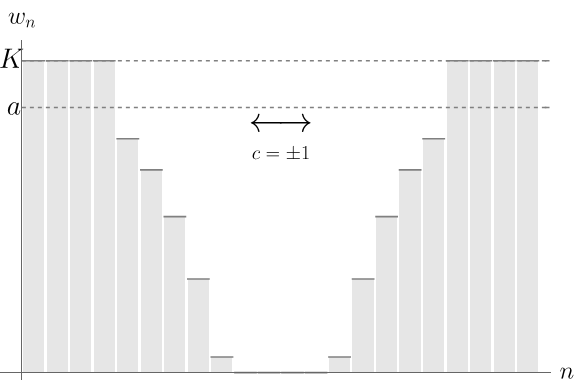}
    \caption{Radial waves from Ex.~\ref{x:radial}. Illustration of Lemmas~\ref{l:right:traveling} and~\ref{l:right:traveling:dec}.}
    \label{fig:ex_radial_wave}
\end{figure}

\end{example}

Our key insight regarding bistable RDCA is the following necessary conditions on the first wave core value $w_1$.

\begin{lemma}[Necessary condition on $w_1$]\label{l:necessary:w1}
Let $\delta, a, K \in \N$ be such that $2\leq a \leq K-2$ and the bistable reaction function $f$ be defined by \eqref{e:bistable}. If there exists a left traveling wave \eqref{e:wave}--\eqref{e:wave:core} of bistable RDCA~\eqref{e:RDCA:DDS} with $c=-1$, then 
\begin{equation}
\label{e:left:NC1}
    \delta\geq w_1 > a
\end{equation}
and
\begin{equation}
\label{e:left:NC2}
    f(w_1)\geq 2w_1.
\end{equation}
\end{lemma}
\begin{proof}   
We observe that the definition of $h_\delta$ \eqref{e:h_delta} implies the inequality
\begin{equation}\label{l:Allee:Ncondition:p}
    h_\delta(f(w_1),0)\leq \frac{f(w_1)}{2}.
\end{equation}

Firstly, we provide the proof for inequality \eqref{e:left:NC1}. The former inequality in \eqref{e:left:NC1},  follows immediately from the definition of $h_\delta$ in Eq.~\eqref{e:h_delta}
\[
w_1=h_\delta(f(w_1),0) \leq \delta.
\]

Let us show the latter inequality in \eqref{e:left:NC1}, i.e., that $w_1>a$ holds. We assume by contradiction that $w_1\leq a$. Then Eq.~\eqref{e:recurrence:general} with $c=1$ implies for $n=0$
\[
w_1=f(w_0)+h_\delta(f(w_{-1}),f(w_0))+h_\delta(f(w_1),f(w_0))=0+0+h_\delta(f(w_1),0)=h_\delta(f(w_1),0).
\]
This is a contradiction with \eqref{l:Allee:Ncondition:p}, since $f(w_1)\leq w_1$ for $w_1\in(0,a]_\Z$ implies
\[
h_\delta(f(w_1),0)\leq \frac{f(w_1)}{2}<w_1.
\]

In order to prove the inequality \eqref{e:left:NC2}, we use the inequality \eqref{l:Allee:Ncondition:p} again to get
\[
w_1=h_\delta(f(w_1),0) \leq \frac{f(w_1)}{2},
\]
which is equivalent to  $f(w_1)\geq 2w_1$.
 
\end{proof}

We can now fully characterize parametric values for which there exist left traveling waves of bistable RDCA. Moreover, we provide a constructive proof.
\begin{theorem}[Necessary and sufficient condition left traveling waves]\label{t:ns:left}
Let $\delta, a, K \in \N$ be such that $2\leq a \leq K-2$ and the bistable reaction function $f$ be defined by \eqref{e:bistable}. There exists a left traveling wave \eqref{e:wave}--\eqref{e:wave:core} of bistable RDCA~\eqref{e:RDCA:DDS} with $c=-1$ if and only if $a<\frac{K}{2}$ and there exists $p\in\left(a,\frac{K}{2}\right]_\mathbb{Z}$ satisfying
\begin{equation}
    \delta\geq p,  \label{e:p:value}  
\end{equation}
and 
\begin{equation}
   f(p)\geq 2p. \label{e:p:function}
\end{equation}
\end{theorem}
\begin{proof}
First, let us assume that the inequalities \eqref{e:p:value}--\eqref{e:p:function} are satisfied and let us construct the left traveling wave.
\subsubsection*{Construction of $w_1$.} First, we find $w_1>0$ which satisfies
\[
w_{1}=f(w_0) + h_\delta(f(w_{-1}),f(w_0)) + h_\delta(f(w_{1}),f(w_0)) = h_\delta(f(w_{1}),0).
\]
For this purpose we define an auxiliary function $\psi_0:[p,K]_\mathbb{Z}\rightarrow \mathbb{Z}$ by
\begin{equation}\label{e:psi0}
 \psi_0(m):=h_\delta(f(m),0) - m.
\end{equation}
Obviously, our goal is to find $w_1$ such that $\psi_0(w_1)=0$. 

For $m=p$ we can use the inequalities \eqref{e:p:value}--\eqref{e:p:function} to obtain
\[
\psi_0(p)=h_\delta(f(p),0) - p  \geq h_\delta(2p,0)-p \geq p-p =0.
\]
Similarly for $m=K$, we obtain the reverse inequality
\[
\psi_0(K)=h_\delta(f(K),0) - K = h_\delta(K,0) - K \leq \frac{K}{2}-K = - \frac{K}{2} < 0.
\]

If $\psi_0(p)=0$, then we can set $w_1=p$ as $w_{1}=h_\delta(f(w_{1}),0)$ is fulfilled. Otherwise, we get $\psi_0(p)>0$ and we determine $w_1$ using
\[
M:=\max \{m\in[p,K)_\mathbb{Z}: \psi_0(m)>0\}.
\]
Then we get
\[
\psi_0(M+1) = h_\delta(f(M+1),0) - M-1 \geq h_\delta(f(M),0) - M-1 \geq \psi_0(M)-1 \geq 0,
\]
which implies that $\psi_0(M+1)=0$ and we set $w_1=M+1$.

\subsubsection*{Construction of $w_{n+1}$, $n\in  \mathbb{N}$.} We use similar principles to construct all other terms of the increasing wave profile \eqref{e:wave:ansatz}--\eqref{e:wave}. Let us fix $n\in\mathbb{N}$ and define the function $\psi_n:[w_n,K]_\mathbb{Z}\rightarrow\mathbb{Z}$ by

\begin{equation}\label{e:psin}
\psi_n(m):=f(w_n) - h_\delta(f(w_{n}),f(w_{n-1})) + h_\delta(f(m),f(w_n)) - m.
\end{equation}

 First, we show that $\psi_n(w_n)\geq 0$. This follows from the following estimates
 \begin{align*}
     \psi_n(w_n) &= f(w_n) - h_\delta(f(w_{n}),f(w_{n-1})) + h_\delta(f(w_{n}),f(w_n)) - w_n \\
         &= f(w_n) - h_\delta(f(w_{n}),f(w_{n-1})) - w_n \\
         &=f(w_n) - h_\delta(f(w_{n}),f(w_{n-1})) - f(w_{n-1})+ \\
         & \quad+h_\delta(f(w_{n-1}),f(w_{n-2}))-h_\delta(f(w_{n}),f(w_{n-1}))\\
         &=f(w_n) -f(w_{n-1}) - 2h_\delta(f(w_{n}),f(w_{n-1}))+ h_\delta(f(w_{n-1}),f(w_{n-2}))\\
         &\geq f(w_n) -f(w_{n-1}) - 2 \cdot \frac{f(w_n) -f(w_{n-1})}{2} + h_\delta(f(w_{n-1}),f(w_{n-2}))\\
         &= h_\delta(f(w_{n-1}),f(w_{n-2})) \\
         &\geq 0.
\end{align*}

 Similarly, we show that $\psi_n(K)\leq 0$:
 \begin{align*}
     \psi_n(K) &= f(w_n) - h_\delta(f(w_{n}),f(w_{n-1})) + h_\delta(f(K),f(w_n)) - K\\
         &\leq f(w_n) + h_\delta(f(K),f(w_{n})) - K\\
         &\leq f(w_n) +\frac{K-f(w_n)}{2} - K\\
         &= \frac{f(w_n)-K}{2}\\
         &\leq 0.
 \end{align*}

 If $\psi_n(w_n)= 0$, we set $w_{n+1}=w_n$ since it satisfies Eq.~\eqref{e:recurrence:general}. Similarly, if $\psi_n(K) = 0$, we set $w_{n+1}=K$. Otherwise, we get $\psi_n(w_n) \geq 0 \geq \psi_n(K)$  and we find $w_{n+1}$ by defining
 \begin{equation}\label{e:proof:Mn}
 M:= \max \{ m\in[w_n,K)_\mathbb{Z}: \psi_n(m)>0 \}.
 \end{equation}
 We then get
 \begin{align*}
 \psi_{n}(M+1) &= f(w_n) - h_\delta(f(w_{n}),f(w_{n-1})) + h_\delta(f(M+1),f(w_n)) - M-1\\
 &\geq  f(w_n) - h_\delta(f(w_{n}),f(w_{n-1})) + h_\delta(f(M),f(w_n)) - M-1\\
 & = \psi_{n}(M)-1\\
 &\geq 0,
 \end{align*}
 which -- combined with Eq.~\eqref{e:proof:Mn} -- implies that $\psi_n(M+1)=0$ and we can set $w_{n+1}=M+1$ since it satisfies Eq.~\eqref{e:recurrence:general}.

\subsubsection*{Correctness of the construction.} Let us show that the above construction is finite and satisfies \eqref{e:wave:core}. We demonstrate that if $w_{n-1}=w_{n}<K$ for some $n\in \N$, then  $w_{n+1}>w_n$ and $w_{n+1}\leq K$.

If $w_{n-1}=w_{n}<K$, then we obtain using Eq.~\eqref{e:psin}
\begin{align*}
    \psi_n(w_n) &= f(w_n) - h_\delta(f(w_{n}),f(w_{n-1})) + h_\delta(f(w_{n}),f(w_n)) - w_n\\
        &= f(w_n) - 0 + 0 - w_n\\
        &=f(w_n) - w_n > 0.
\end{align*}
Thus, $w_{n+1}>w_n$ holds. 

Let us prove that $w_{n+1}\leq K$ for all $w_n,w_{n-1}$. We can use Eq.~\eqref{e:psin} to estimate
\begin{align*}
    \psi_n(w_n+1) &= f(w_n) - h_\delta(f(w_{n}),f(w_{n-1})) + h_\delta(f(w_{n}+1),f(w_n)) - w_n - 1 \\
        &= f(w_n) - 0 + 0 - w_n - 1\\
        &=f(w_n) - w_n - 1 \geq 0.
\end{align*}
Consequently, $w_{n+1}\in [w_n+1,K]_\Z$ holds.

 To conclude the proof, we observe that Lem.~\ref{l:necessary:w1} imply that there is no wave profile \eqref{e:wave}--\eqref{e:wave:core} if inequalities \eqref{e:p:value}--\eqref{e:p:function} are violated for all $p\in\left(a,\frac{K}{2}\right]_\mathbb{Z}$.
\end{proof}

Analogously we immediately obtain the same statement for the right traveling wave.

\begin{theorem}[Necessary and sufficient condition right traveling waves]\label{t:ns:right}
Let $\delta, a, K \in \N$ be such that $2\leq a \leq K-2$ and the bistable reaction function $f$ be defined by \eqref{e:bistable}. There exists a right traveling wave \eqref{e:wave}--\eqref{e:wave:core} of bistable RDCA~\eqref{e:RDCA:DDS} with $c=1$ if and only if $a>\frac{K}{2}$ and there exists $p\in\left[\frac{K}{2},a\right)_\mathbb{Z}$ satisfying
\begin{equation}
    \delta\geq K-p,  \label{e:right:p:value}  
\end{equation}
and 
\begin{equation}
   f(p)\leq 2p-K. \label{e:right:p:function}
\end{equation}
\end{theorem}

Despite the fact that we get full characterization of left and right traveling waves in Thms.~\ref{t:ns:left}--\ref{t:ns:right} as well as a constructive proof, the detailed inspection of the proof shows that the construction is not necessarily unique (there could be more solutions of equations $\psi_n(m)=0$). The following examples illustrate the fact that numerous traveling fronts naturally coexist in this setting.

\begin{example}[Multiplicity of traveling waves]\label{ex:nonuniqueness}
Let us consider $a=3$, $K=30$ and the bistable nonlinearity \eqref{e:bistable} defined by the following table (and depicted in Fig.~\ref{fig:ex_nonuniqueness_nonexistence_reaction})

\begin{figure}
    \centering
    \includegraphics[width=.45\textwidth]{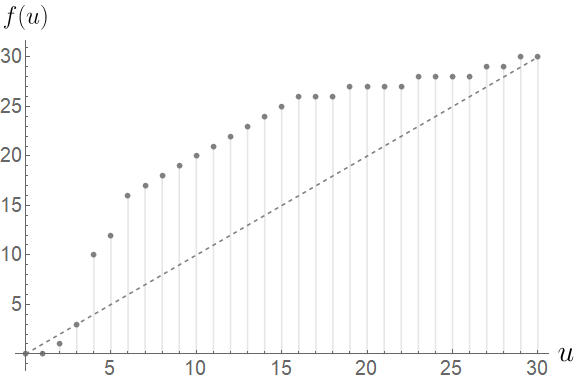}
    \includegraphics[width=.45\textwidth]{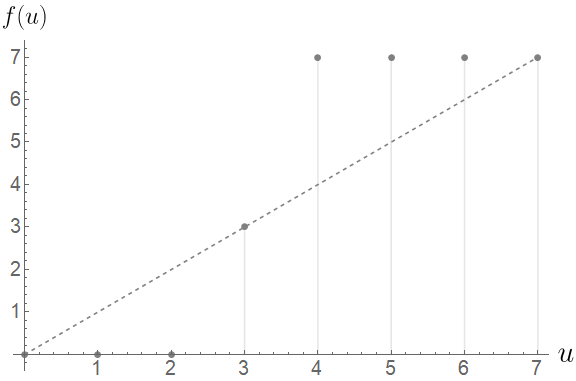}
    \caption{The bistable reaction function from Example~\ref{ex:nonuniqueness} (left panel) and reaction function~\eqref{e:bistable:x:nonexistence} from Example~\ref{x:nonexistence} for which there are neither traveling nor pinned traveling profiles for certain diffusion parameters (right panel).}
    \label{fig:ex_nonuniqueness_nonexistence_reaction}
\end{figure}

{
\[
\begin{smallmatrix}
 m & & 0 & 1 & 2 & 3 & 4 & 5 & 6 & 7 & 8 & 9 & 10 & 11 & 12 & 13 & 14 & 15 & 16 & 17 & 18 & 19 & 20 & 21 & 22 & 23 & 24 & 25 & 26 & 27 & 28 & 29 & 30 \\
 f(m) & & 0 & 0 & 1 & 3 & 10 & 12 & 16 & 17 & 18 & 19 & 20 & 21 & 22 & 23 & 24 & 25 & 26 & 26 & 26 & 27 & 27 & 27 & 27 & 28 & 28 & 28 & 28 & 29 & 29 & 30 & 30 \\
\end{smallmatrix}
\]
}
Apparently, the assumptions of Thm.~\ref{t:ns:left} are satisfied, e.g., for $p=4$ and $\delta\geq p=4>a=3$, since $f(4)=10>2\cdot4=8$. Let us consider, e.g., $\delta=5$.

In order to find $w_1$, we make use of the function $\psi_0$ given by Eq.~\eqref{e:psi0} and we get
\[
\begin{smallmatrix}
 m & & 4 & 5 & 6 & 7 & 8 & 9 & 10 & 11 & 12 & 13 & 14 & 15 & 16 \
& 17 & 18 & 19 & 20 & 21 & 22 & 23 & 24 & 25 & 26 & 27 & 28 & 29 & 30 \
\\
 \psi_0(m) & & 1 & 0 & -1 & -2 & -3 & -4 & -5 & -6 & -7 & -8 & -9 & \
-10 & -11 & -12 & -13 & -14 & -15 & -16 & -17 & -18 & -19 & -20 & -21 \
& -22 & -23 & -24 & -25 \\
\end{smallmatrix}
\]
Therefore, we set $w_1=5$. Using Eq.~\eqref{e:psin}, we compute values of $\psi_1$
\[
\begin{smallmatrix}
 m & & 5 & 6 & 7 & 8 & 9 & 10 & 11 & 12 & 13 & 14 & 15 & 16 & 17 & 18 & 19 & \
20 & 21 & 22 & 23 & 24 & 25 & 26 & 27 & 28 & 29 & 30 \\
 \psi_1(m) & & 2 & 3 & 2 & 2 & 1 & 1 & 0 & 0 & -1 & -2 & -3 & -4 & -5 & -6 & -7 & \
-8 & -9 & -10 & -11 & -12 & -13 & -14 & -15 & -16 & -17 & -18 \\
\end{smallmatrix}
\]
We set $w_2=11$ but observe at the same time that we could choose $w_2=12$. We also remark that $\psi_n$ is not decreasing in general (cf. with values $\psi_1(5), \psi_1(6), \psi_1(7)$). Continuing in this manner we can arrive to four left traveling waves (depicted in~Fig.~\ref{fig:ex_nonuniqueness_waveprofiles})

\begin{figure}
    \centering
    \includegraphics[width=.4\textwidth]{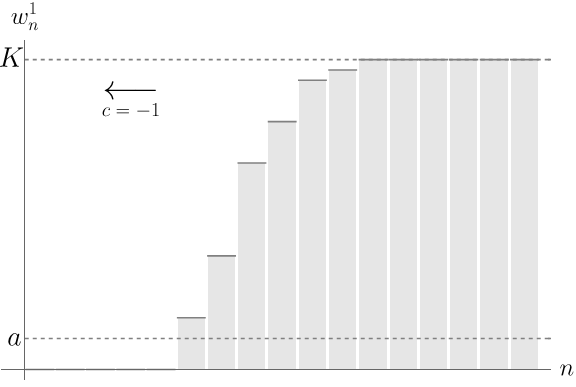}
    \includegraphics[width=.4\textwidth]{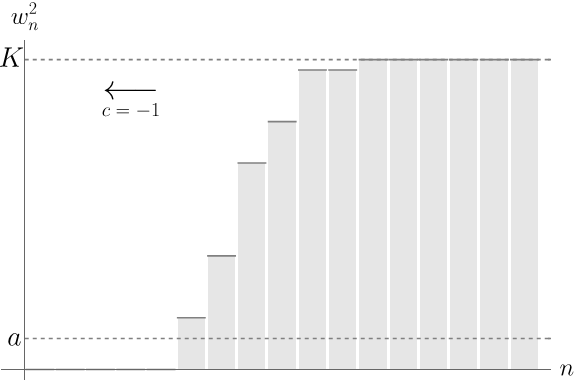}
    \includegraphics[width=.4\textwidth]{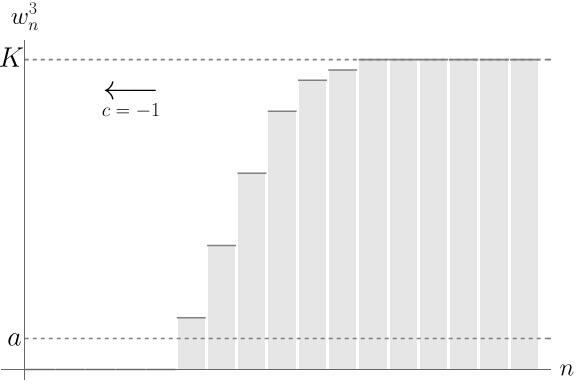}
    \includegraphics[width=.4\textwidth]{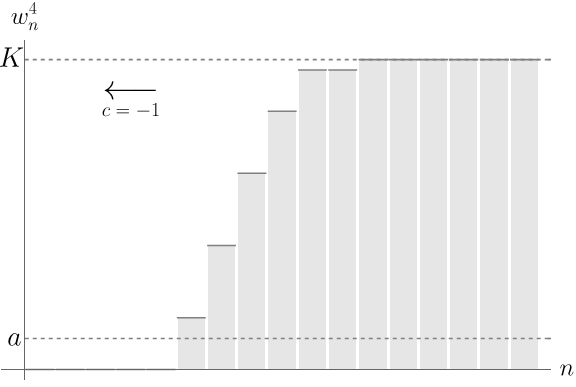}
    \caption{Illustration of nonuniqueness of traveling waves. Four different profiles from Example~\ref{ex:nonuniqueness}.}
    \label{fig:ex_nonuniqueness_waveprofiles}
\end{figure}

\begin{align*}
    w^1 &= (\ldots, 0, 0, 0, 5, 11, 20, 24, 28, 29, 30, 30, 30, \ldots),\\
    w^2 &= (\ldots, 0, 0, 0, 5, 11, 20, 24, 29, 29, 30, 30, 30, \ldots),\\
    w^3 &= (\ldots, 0, 0, 0, 5, 12, 19, 25, 28, 29, 30, 30, 30, \ldots),\\
    w^4 &= (\ldots, 0, 0, 0, 5, 12, 19, 25, 29, 29, 30, 30, 30, \ldots).
\end{align*}
\end{example}


\section{Pinned waves}\label{s:pinned}
Thms.~\ref{t:ns:right}--\ref{t:ns:left} immediately imply necessary and sufficient conditions for the nonexistence of left or right traveling waves. There is a naturally intriguing question, whether we can find a pinned wave (with $c=0$) whenever their assumptions are violated.

Surprisingly, this answer, which is positive in the case of lattice \cite{Logan1994} or partial \cite{Chow1998} differential equations, can be negative in the setting of cellular automata. We provide a simple example to illustrate this fact before we discuss the case in which neither traveling nor pinned fronts exist in general.

\begin{example}\label{x:nonexistence}
Let us consider a bistable nonlinearity \eqref{e:bistable} with $a=3$, $K=7$ given by (see Fig.~\ref{fig:ex_nonuniqueness_nonexistence_reaction})


\begin{equation}\label{e:bistable:x:nonexistence}
f(u)=\begin{cases}
    0 &\text{if } u\in [0,3)_\Z, \\
    3 & \text{if } u=3, \\
    7 &\text{if } u\in (3,7]_\Z. \\    
    \end{cases}
\end{equation}
Thms.~\ref{t:ns:left} and~\ref{t:ns:right} imply there is no traveling wave for $\delta=3$. Let us show that there is no pinned wave \eqref{e:wave}--\eqref{e:wave:core} with $c=0$ either.

Since $w_{-1}=w_0=0$ we have from \eqref{e:recurrence:general}
\[
w_0=0 = f(w_0) + h_3(f(w_{-1}),f(w_0)) + h_3(f(w_{1}),f(w_0)) = 0 + 0 + h_3(f(w_1),0).
\]
Consequently, we have either $w_1=1$ or $w_1=2$. If $w_1=1$ we have from \eqref{e:recurrence:general}
\[
w_1= 1 = f(w_1) + h_3(f(w_{0}),f(w_1)) + h_3(f(w_{2}),f(w_1)) = 0 + 0 + h_3(f(w_2),0),
\]
which is only possible if $w_2=3$. However, since $\delta=3$ the equation
\[
w_2= 3 = f(w_2) + h_3(f(w_{1}),f(w_2)) + h_3(f(w_{3}),f(w_2)) = 3 - 1 + h_3(f(w_3),3)
\]
cannot be satisfied for any $w_3$. If $w_1=2$, then
\[
w_1= 2 = f(w_1) + h_3(f(w_{0}),f(w_1)) + h_3(f(w_{2}),f(w_1)) = 0 + 0 + h_3(f(w_2),0)
\]
cannot be satisfied since $h_3(f(3),0)=h_3(3,0)=1$ and $h_3(f(n),0)=h_3(7,0)=3$ for $n\in(3,7]_\Z$. Consequently, there is no pinned wave.
\end{example}

We will later see that the configuration from Example~\ref{x:nonexistence} leads to higher-order waves. However, we return first to the pinned waves waves with $c=0$. The pinned waves always exist in the case of the weakest diffusion $\delta=1$.

\begin{lemma}[Pinned waves for $\delta=1$]\label{l:pinned:trivial}
    Let $\delta, a, K \in \N$ be such that $2\leq a \leq K-2$ and the bistable reaction function $f$ be defined by \eqref{e:bistable}. If  $\delta=1$, then
    \begin{enumerate}
        \item there are no moving traveling waves of \eqref{e:RDCA:DDS},
        \item there exists a pinned traveling wave of \eqref{e:RDCA:DDS}.
    \end{enumerate}
\end{lemma}
\begin{proof}
    The former statement follows immediately from Thms.~\ref{t:ns:left}--\ref{t:ns:right}. We obtain the latter one by observing that 
    \[
        w=(\dots,0,0,1,K-1,K,K,\dots)
    \]
    is a pinned wave for $\delta=1$. 
\end{proof}

Once $\delta>1$, the situation becomes more intricate and we can take advantage of the following necessary and sufficient condition.

\begin{lemma}[Necessary and sufficient condition for pinned waves]\label{l:pinned:ns}
Let $\delta, a, K \in \N$ be such that $2\leq a \leq K-2$ and the bistable reaction function $f$ be defined by \eqref{e:bistable}. A wave profile $w$ \eqref{e:wave}--\eqref{e:wave:core} is pinned if and only if 
\begin{equation}\label{e:pinned:sum}
    h_\delta(f(w_n),f(w_{n-1}))= \sum_{i=1}^{n-1} \left(w_i-f(w_i)\right)
\end{equation}
for all $n \in [1,+\infty)_\Z$.
\end{lemma}
\begin{proof}
Let us demonstrate the necessity first. Using Eq.~\eqref{e:recurrence:general}, we obtain
\begin{align*}
    w_0&=f(w_0)+h_\delta(f(w_{-1}),f(w_{0}))+h_\delta(f(w_1),f(w_{0}))\\
    0 &= 0 + 0 + h_\delta(f(w_1),f(w_{0}))= h_\delta(f(w_1),f(w_{0})).    
\end{align*}
Consequently,
\begin{align*}
    w_1&=f(w_1)+h_\delta(f(w_0),f(w_{1}))+h_\delta(f(w_2),f(w_{1}))= f(w_1) + 0 + h_\delta(f(w_2),f(w_{1}))\\    
    h_\delta(f(w_2),f(w_{1})) &= w_1-f(w_1)
\end{align*}
and 
\begin{align*}
    w_2&=f(w_2)+h_\delta(f(w_1),f(w_{2}))+h_\delta(f(w_3),f(w_{2}))\\
    w_2 &= f(w_2) - (w_1-f(w_1)) + h_\delta(f(w_3),f(w_{2}))\\
    h_\delta(f(w_3),f(w_{2})) &= w_2-f(w_2)+w_1-f(w_1)=\sum_{i=1}^{2} \left( w_i-f(w_i)\right).
\end{align*}
Applying similar arguments for $w_n$ for all $n \in \N$, we obtain Eq.~\eqref{e:pinned:sum}.

To prove the sufficiency, it is enough to realize that if for all $n \in \N$ and $h_\delta(f(w_n),f(w_{n-1}))$ is given by Eq.~\eqref{e:pinned:sum}, then Eq.~\eqref{e:recurrence:general} with $c=0$ is satisfied for all $n \in \N$.
\end{proof}

Lemma~\ref{l:pinned:ns} immediately restricts the values $w_1$ and $w_{N-1}$ of the wave core for pinned waves.

\begin{corollary}[Core boundaries]\label{c:core:boundaries}
Let $\delta, a, K \in \N$ be such that $2\leq a \leq K-2$ and the bistable reaction function $f$ be defined by \eqref{e:bistable}. If there exists a pinned wave \eqref{e:wave}--\eqref{e:wave:core} with $c=0$, then
\begin{equation}\label{e:pinned:NC1}
    f(w_1) \leq 1 \leq w_1 <a <w_{N-1} \leq K-1 \leq f(w_{N-1}).
\end{equation}
\end{corollary}
\begin{proof}
The first inequality in~\eqref{e:pinned:NC1} follows immediately from Eq.~\eqref{e:pinned:sum}, which yields
\[h_\delta(f(w_{1}),f(w_{0}))=0.\]
The inequality $K-1\leq f(w_{N-1})$ can be obtained in a similar fashion
\begin{align*}
    K=w_N &= f(w_N) + h_\delta(f(w_{N-1}),f(w_N))+h_\delta(f(w_{N+1}),f(w_N))=K+h_\delta(f(w_{N-1}),f(w_N))+0,
\end{align*}
which implies $h_\delta(f(w_{N-1}),f(w_N))=0$ and $K-1 \leq f(w_{N-1})$.

To complete the proof, we observe that the inner inequalities in Eq.~\eqref{e:pinned:NC1} follow from the definition of $f$ in \eqref{e:bistable}.
\end{proof}

Moreover, we can show that pinned waves have strictly increasing cores.

\begin{lemma}[Strictly increasing cores of pinned waves]\label{l:pinned:unequal}
Let $\delta, a, K \in \N$ be such that $2\leq a \leq K-2$ and the bistable reaction function $f$ be defined by \eqref{e:bistable}. If there exists a pinned wave \eqref{e:wave}--\eqref{e:wave:core} with $c=0$, then
\begin{equation}\label{e:pinned:unequal}
    w_{n-1} < w_{n}, \quad \text{ for all } n\in(0,N]_\Z.
\end{equation}
\end{lemma}
\begin{proof}
Assume by contradiction that $0<w_{n-1}=w_n<K$ for some $n \in (0,N]_\Z$. Cor.~\ref{c:core:boundaries} implies that the core contains at least one value $w_i\in(0,a)_\Z$ and at least one value $w_i\in(a,K)_\Z$. Since $w_i-f(w_i)>0$ for $w_i \in (0,a)_\Z$, then Lem.~\ref{l:pinned:ns} implies that 
\[
h_\delta(f(w_{n}),f(w_{n-1}))=\sum_{i=1}^{n-1} \left( w_i-f(w_i) \right)>0
\]
is strictly increasing for $n=2,\ldots,\ell,\ell+1$ where 
\[
\ell = \max\{i: w_i<a \}.
\]
Similarly, $h_\delta(f(w_{n}),f(w_{n-1}))$ is strictly decreasing for $n=\ell+1,\ldots,N-1$ if $w_{\ell+1}>a$ or for $n=\ell+2,\ldots,N-1$ if $w_{\ell+1}=a$  since $w_i-f(w_i)<0$ for $w_i \in (a,K)_\Z$. However, Cor.~\ref{c:core:boundaries} and Lem.~\ref{l:pinned:ns} imply
\[
h_\delta(f(w_{N}),f(w_{N-1})=\sum_{i=1}^{N-1} \left( w_i-f(w_i) \right)=0.
\]
Consequently, $h_\delta(f(w_{n}),f(w_{n-1}))>0$ for all $n=2,\ldots,N-1$, a contradiction.
\end{proof}

Consequently, we can find a straightforward bound on the length of pinned waves' cores $w^\mathcal{C}$ and show that small diffusion implies steep cores.

\begin{corollary}[Core length of pinned waves]\label{c:pinned:bound}
Let $\delta, a, K \in \N$ be such that $2\leq a \leq K-2$ and the bistable reaction function $f$ be defined by \eqref{e:bistable}. If there exists a pinned wave \eqref{e:wave}--\eqref{e:wave:core} with $c=0$, then
\begin{equation}\label{e:pinned:bound}
    N \leq 2\delta + 2.
\end{equation}
\end{corollary}
\begin{proof}
Lem.~\ref{l:pinned:unequal} yields that the core of pinned waves is strictly increasing. Moreover, combining Eq.~\eqref{e:pinned:sum} with the definition of $h_\delta$ \eqref{e:h_delta}, we obtain
\[
\sum_{i=1}^{n-1} \left( w_i-f(w_i)\right)=h_\delta(f(w_n),f(w_{n-1})\leq \delta.
\]
Since $w_i-f(w_i)>0$ for $w_i \in (0,a)_\Z$, we have 
\[
\lvert \lbrace w_i\in\wc: 0 < w_i < a\rbrace\rvert \leq \delta.
\]
Similarly, we obtain
\[
\lvert \lbrace w_i\in\wc: a < w_i < K\rbrace\rvert \leq \delta.
\]
Since $a$ may also be in the wave core and there are $N-1$ elements in the wave core, we have
\[
N-1 \leq 2\delta +1,
\]
which implies \eqref{e:pinned:bound}.
\end{proof}

\begin{remark}
The bound~\eqref{e:pinned:bound} is optimal, since there is always a pinned wave with $\wc=(1,a,K-1)$ when $a=K/2$ and $\delta=1$.    
\end{remark}













\section{Higher-order traveling waves}\label{s:higher}




In this section we show that parametric regions in which neither moving nor pinned traveling waves exist lead to higher-order traveling waves. These solutions travel but periodically change their profiles.
\begin{definition}
    A profile $w$ forms a  higher-order $(c,m)$-traveling wave for some $c \in \Z, m \in \N$  if
    \[
    (\dots, w_{-1-c}, w_{-c}, w_{1-c}, \dots) = F^m\left((\dots, w_{-1}, w_{0}, w_{1}, \dots)\right).
    \]
    
    \noindent The ratio     
    \begin{equation}\label{e:gamma}
        \gamma=\frac{c}{m}.    
    \end{equation}    
    is called the speed of $(c,m)$-traveling wave. 
\end{definition}

In other words, a  higher-order $(c,m)$-traveling wave is a set of $m$ profiles which periodically repeat themselves and each periodic occurrence is coupled with a shift/speed $c$. Note that we do not call them periodic waves because these are commonly used to periodic profiles which travel \cite{Sherratt1996}. 

To illustrate, let us go back to Example~\ref{x:nonexistence} and show the existence of a higher-order $(1,2)$-traveling wave.

\begin{example}\label{x:nonexistence:1-2}
Let us consider again a bistable nonlinearity~\eqref{e:bistable} with $a=3$, $K=7$ given by~\eqref{e:bistable:x:nonexistence}, see Fig.~\ref{fig:ex_nonuniqueness_nonexistence_reaction}, 
and diffusion  $\delta=3$. We have shown in Ex.~\ref{x:nonexistence} that in this case there are neither moving nor pinned traveling waves of the form~\eqref{e:wave:ansatz}.

The following computation shows that in this parametric setting the profile
\[
w=(\ldots, 0,0,3,4,7,7,\ldots)
\]
forms higher-order $(1,2)$-traveling wave. Indeed, we have
\begin{align*}
w &=(\ldots, 0,0,3,4,7,7,\ldots) \\
 \overline{w}=F(w) &= (\ldots, 0, 1, 4, 5, 7, 7, \ldots),\\    
 F^2(w)=F(\overline{w}) &= (\ldots, 0,3,4,7,7,7, \ldots), \\
\end{align*}
see Fig.~\ref{fig:example_higher_order_waves} for illustration. Note that profiles $w$ with a core $w^\mathcal{C}=(3,4)$ and $\overline{w}$ with a core $\overline{w}^\mathcal{C}=(1,4,5)$ interchange periodically every $m=2$ periods while being shifted by $c=-1$ to the left. They both form a higher-order $(-1,2)$-traveling wave with speed $\gamma=-1/2$.

\begin{figure}
    \centering
    \includegraphics[width=.32\textwidth]{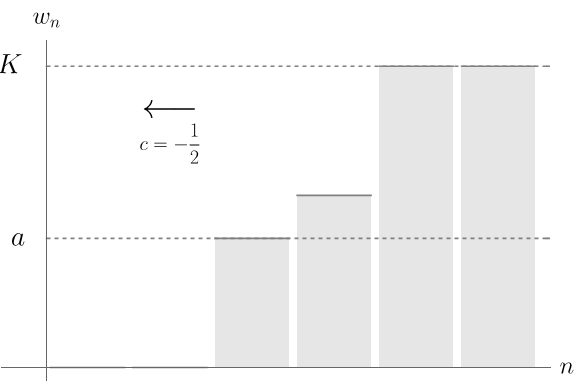}
    \includegraphics[width=.32\textwidth]{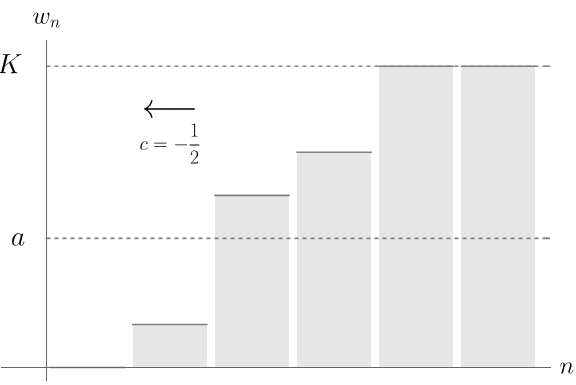}
    \includegraphics[width=.32\textwidth]{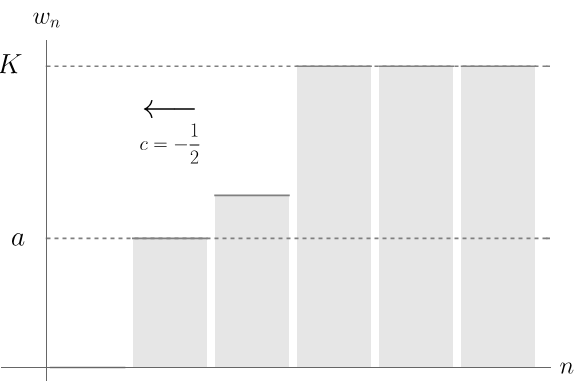}
    \caption{The higher-order $(1,2)$-traveling wave from Ex.~\ref{x:nonexistence} with $K=7$, $a=3$. The profiles with cores $w^\mathcal{C}=(3,4)$ (the first and third panel) and $\overline{w}^\mathcal{C}=(1,4,5)$ (the central panel) interchange periodically every $m=2$ periods while being shifted by $c=1$ to the left.}
    \label{fig:example_higher_order_waves}
\end{figure}

Note that $w$ is a higher-order $(-1,2)$-traveling wave for any $\delta\geq a=3$.
\end{example}

This example can be generalized so that we can get behavior of waves for any diffusion parameter and the extreme case of reaction function like 
\eqref{e:bistable:x:nonexistence}, which we call the `maximal' bistable reaction and can be seen as an equivalent of bistable caricatures~\cite{McKean1970} 
\begin{equation}\label{e:maximal:reaction}
f^{a,K}_{\mathrm{max}} (u) :=\begin{cases}
    0, & u <a, \\
    a, & u=a, \\
    K, & u>a.    
\end{cases}    
\end{equation}

In this case, we can mimic the construction from Ex.~\ref{x:nonexistence:1-2} to get the existence of higher-order $(\pm 1,2)$-traveling waves in general.

\begin{lemma}[Sufficient condition for existence $(\pm 1,2)$-traveling waves]\label{l:suf:higher:order}
Let $1 < a < \frac{K}{2}$ and $\delta$ be such that 
$h_\delta(K,0)=a$. Then $w=(\ldots, 0,0,a,K-a,K,K,\ldots) $ is a higher-order $(-1,2)$-traveling wave of \eqref{e:RDCA:DDS} with $f^{a,K}_{\mathrm{max}}$.

Let $\frac{K}{2}<a<K-1$ and $\delta$ be such that 
$h_\delta(K,0)=a$. Then $w=(\ldots, 0,0,K-a,a,K,K,\ldots) $ is a higher-order $(1,2)$-traveling wave of \eqref{e:RDCA:DDS} with $f^{a,K}_{\mathrm{max}}$.
\end{lemma}

\begin{proof}
    Let us prove the former statement, the latter can be proved analogously. We first note that $h_\delta(K,0)=a$ implies $h_\delta(K,a)>\lfloor \frac{a}{2} \rfloor$, see~\eqref{e:h_delta}.  Let $1 < a < \frac{K}{2}$ and $\delta$ be such that $h_\delta(K,a)>\lfloor \frac{a}{2} \rfloor$ and $h_\delta(K,0)=a$. Note that $a<\frac{K}{2}$ implies $h_\delta(a,0)=\lfloor \frac{a}{2} \rfloor$. Applying $F$ from \eqref{e:RDCA:F} to $w=(\ldots, 0,0,a,K-a,K,K,\ldots) $, we get
    \[
    \overline{w}=F(w) = \left(\ldots, 0, \left\lfloor \frac{a}{2} \right\rfloor, a-\left\lfloor \frac{a}{2} \right\rfloor+h_\delta(K,a), K-h_\delta(K,a), K, K, \ldots\right)
    \]
    Consequently,
    \[
    F^2(w)=F(\overline{w}) = (\ldots, 0,a,K-a,K,K,K\ldots).\qedhere
    \]
\end{proof}

We can now fully characterize waves for automata \eqref{e:RDCA:DDS} with $f^{a,K}_{\mathrm{max}}$ given by \eqref{e:maximal:reaction} and reveal that there are three qualitatively different types of wave behavior for capacities $K$ being either even, or odd of the type $K=4m-1$ or $K=4m+1$, $m\in \N$.

\begin{corollary}[Characterization of bistable RDCAs with `maximal' reaction]\label{c:char:maximal}
    Let us consider the cellular automaton \eqref{e:RDCA:DDS} with $f^{a,K}_{\mathrm{max}}$ given by \eqref{e:maximal:reaction} and $2\leq a\leq K-2$.
     \begin{enumerate}
        \item If $\delta \leq \min \lbrace a-1, K-a-1 \rbrace$, there exists a pinned wave $w^*$.

        \item If $a \in \left\{\lfloor \frac{K}{2}\rfloor,\lceil  \frac{K}{2}\rceil\right\}$ and
        \begin{enumerate}
            \item $K = 4m-1$, $m \in \N$, there exists a higher order $(\pm 1, 2)$-traveling wave for all $\delta \geq \min\lbrace a, K-a\rbrace$.
            \item $K \neq 4m-1$, $m \in \N$, there exists a pinned wave $w^*$ for all $\delta \geq \min\lbrace a, K-a\rbrace$.
        \end{enumerate}
                
        \item If $a \not\in \left\{\lfloor \frac{K}{2}\rfloor,\lceil  \frac{K}{2}\rceil\right\}$ and
        \begin{enumerate}
            \item $\delta = a$, then there exists a higher-order $(\pm 1, 2)$-traveling wave.
            \item $\delta\geq \min\lbrace a+1, K-a+1\rbrace$, then there exists a traveling wave with $c=\pm 1$.
        \end{enumerate}        
    \end{enumerate}
\end{corollary}

\begin{proof}
Cases 2(a) and 3(a) follow from Lem.~\ref{l:suf:higher:order}. The case 3(b) is a consequence of Thm.~\ref{t:ns:left}. We prove the rest constructively. To prove case 1, observe that the profile
\[
w_1^*=(\ldots,0,0,\delta,K-\delta,K,K,\ldots)
\]
is a pinned wave as $F(w_1^*)=w_1^*$. Similarly, to prove case 2(b) observe that the profile
\[
w_2^*= (\ldots,0,0,m,a,K-m,K,K,\ldots)
\]
is pinned if $a \in \left\{\lfloor \frac{K}{2}\rfloor,\lceil  \frac{K}{2}\rceil\right\}$ and $K=4m+1$, $m \in \N$. Finally, the profile
\[
w_3^*= (\ldots,0,0,\left\lfloor\frac{K}{4}\right\rfloor,a,K-\left\lfloor\frac{K}{4}\right\rfloor,K,K,\ldots)
\]
is pinned if  $a=\frac{K}{2}$ and $K=2m$, $m \in \N$.\qedhere



\end{proof}

\begin{figure}
    \centering
    \includegraphics[width=0.32\textwidth]{pic_CAwaves_ctverecky_n7_delta.pdf}
    \includegraphics[width=0.32\textwidth]{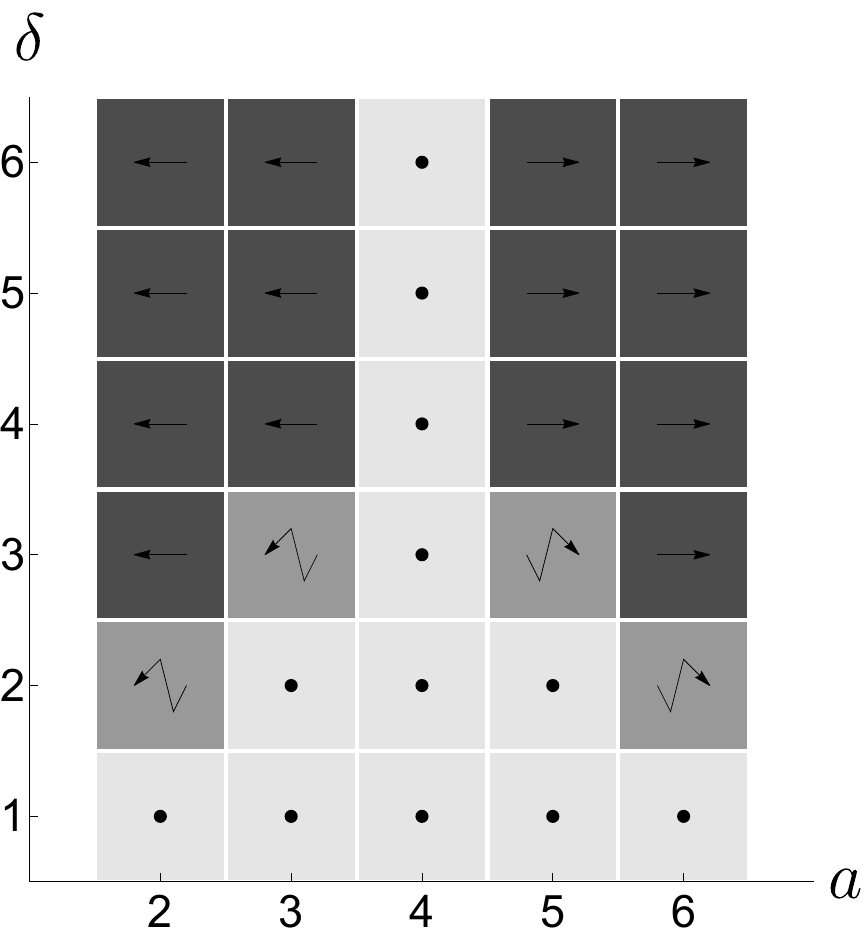}
    \includegraphics[width=0.32\textwidth]{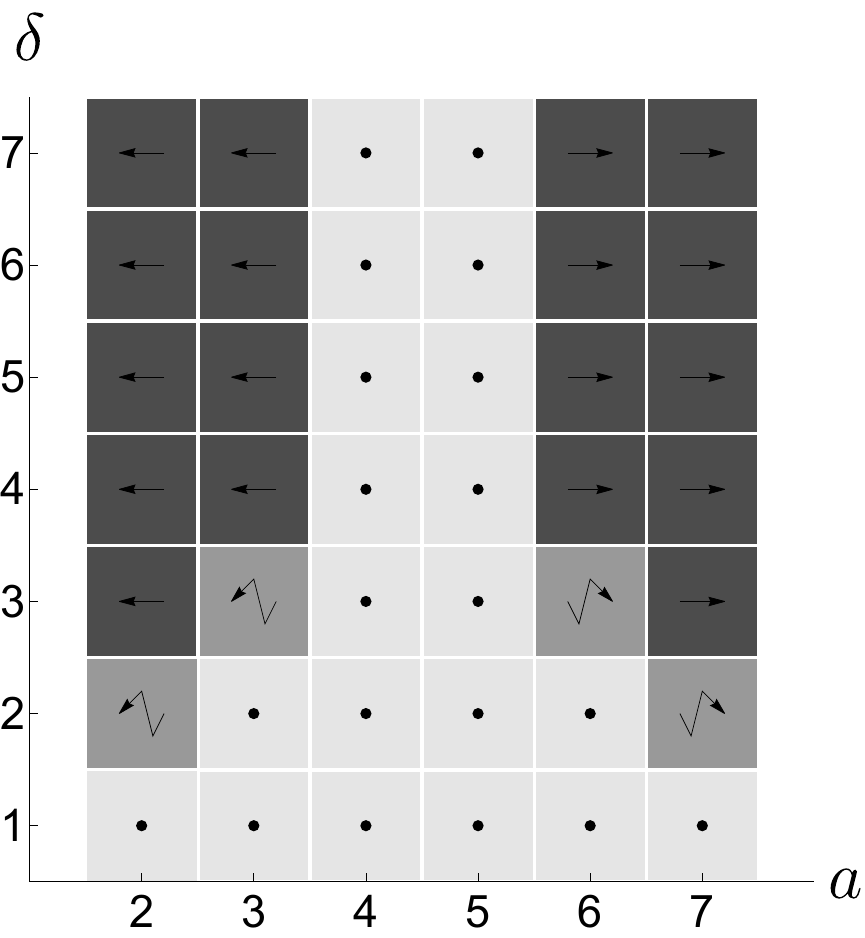}
    \caption{Existence of higher-order traveling waves in bistable automata with the bistable reaction function~\eqref{e:maximal:reaction} for $K=7$ (left panel), $K=8$ (central panel), and $K=9$ (right panel). Combinations of viability $a$ and diffusion parameter $\delta$ with light gray cells and a dot correspond to pinned waves, dark gray cells with straight arrows to monotone moving traveling waves~\ref{e:wave:ansatz}, and medium gray with  zigzag arrows to higher-order $(\pm 1,2)$-waves.}
    \label{fig:3scen}
\end{figure}

If we consider the maximal reaction function $f^{a,K}_{\mathrm{max}}$ from \eqref{e:maximal:reaction}, Cor.~\ref{c:char:maximal} yields three qualitatively different cases.
\begin{example}\label{x:3cases}
    We focus on the most important differences which correspond to the behavior for balanced or almost balanced reaction with $a \in \left\{\lfloor \frac{K}{2}\rfloor,\lceil  \frac{K}{2}\rceil\right\}$.

    If $K = 4m-1$, $m \in \N$, there exist infinitely many pairs $(a,\delta)$ leading to higher-order waves for all values with $a \in \left\{\lfloor \frac{K}{2}\rfloor,\lceil  \frac{K}{2}\rceil\right\}$ and $\delta \geq \left\{a, K-a\right\}$, see the left panel of Fig.~\ref{fig:3scen} with $K=7$.

    If $K = 2m$, $m \in \N$, there exist finitely many pairs $(a,\delta)$ leading to higher-order waves since the pairs with $a=K/2$ lead to pinned waves for any $\delta\in\mathbb{N}$, see the central panel of Fig.~\ref{fig:3scen} with $K=8$.

    Similarly if $K = 4m+1$, $m \in \N$, there exist finitely many pairs $(a,\delta)$ leading to higher-order waves since the pairs with $a \in \left\{\lfloor \frac{K}{2}\rfloor,\lceil  \frac{K}{2}\rceil\right\}$ lead to pinned waves for any $\delta\in\mathbb{N}$, see the right panel of Fig.~\ref{fig:3scen} with $K=9$.        
\end{example}




\section{Examples of traveling waves}\label{s:examples:TW}
In this final section we apply our results and discuss the impact of the reaction function's steepness on the existence of moving, pinned, and higher-order waves. For this purpose, we consider a class of bistable reaction functions
\begin{equation}\label{e:truncated:f}
    f(p) = \begin{cases}
    \min \{K, \lceil {p+\lambda p(K-p)(p-a)} \rceil \},\quad \text{if } p>a,\\
    \max\{0,\lfloor {p+\lambda p(K-p)(p-a)} \rfloor \},\quad \text{if } p\leq a.
    \end{cases}
\end{equation}
with a parameter $\lambda\in\R$. This form is chosen to satisfy \eqref{e:bistable}. Moreover, for small values of $\lambda$ we get the `weakest' bistability of the shape as in the left panel in Fig.~\ref{fig:reaction_waveprofile_comparison}. On the contrary, if $\lambda$ is large enough, we get a Heaviside-like `maximal' bistability as in Fig.~\ref{fig:ex_nonuniqueness_nonexistence_reaction}. We call $f$ in \eqref{e:truncated:f} the truncated polynomial bistable reaction function and illustrate its shape on six examples in Fig.~\ref{fig:x:polynomial:rf}.

\begin{figure}
    \centering
    \subcaptionbox{$\lambda=0.1$}{\includegraphics[width=0.32\textwidth]{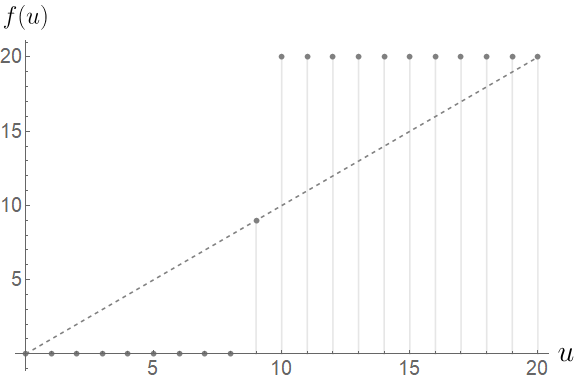}}    
    \subcaptionbox{$\lambda=0.05$}{\includegraphics[width=0.32\textwidth]{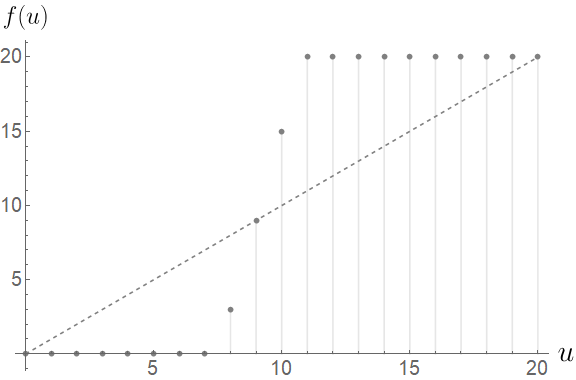}}
    \subcaptionbox{$\lambda=0.02$}{\includegraphics[width=0.32\textwidth]{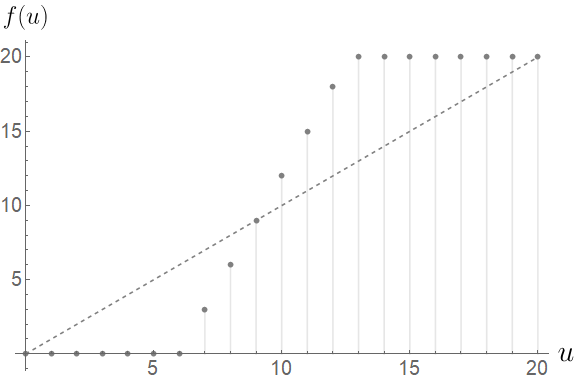}}
    \subcaptionbox{$\lambda=0.01$}{\includegraphics[width=0.32\textwidth]{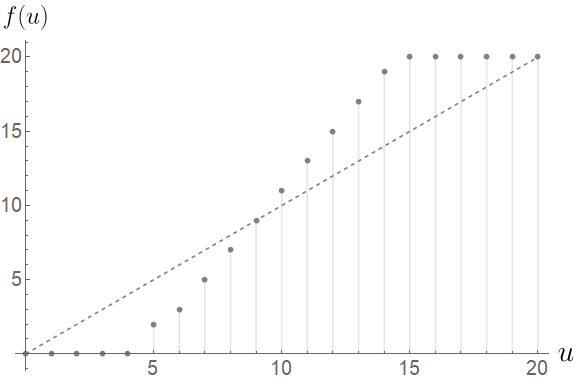}}
    \subcaptionbox{$\lambda=0.005$}{\includegraphics[width=0.32\textwidth]{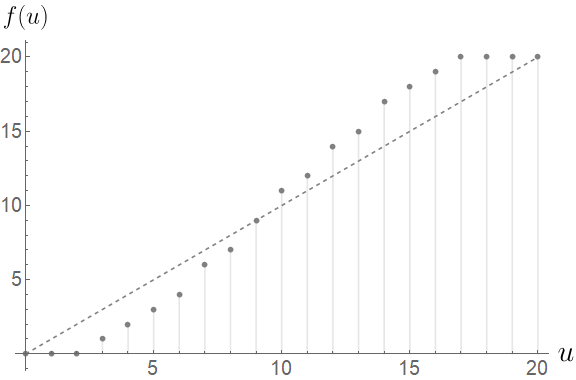}}
    \subcaptionbox{$\lambda=0.001$}{\includegraphics[width=0.32\textwidth]{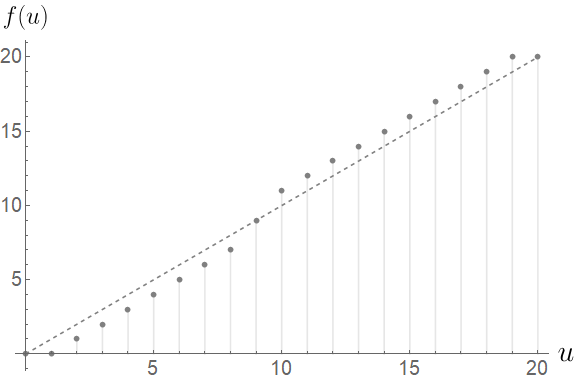}}
    \caption{Truncated polynomial bistable reaction function~\eqref{e:truncated:f} for six various values of $\lambda$ and  $K=20$.}
    \label{fig:x:polynomial:rf}
\end{figure}

We can apply the results from previous sections and derive the existence of thresholds $ \overline{\lambda}$, which yield the existence of moving traveling waves for sufficiently strong diffusion $\delta>0$, and $\underline{\lambda}$, for which there are no moving traveling waves for any diffusion $\delta>0$.

\begin{theorem}[Existence of moving traveling waves for the truncated polynomial bistable reaction function]\label{t:truncated:f}
Let $\delta \in \N$, $a,K \in \N$ be such that $2\leq a \leq K-2$ and the truncated polynomial bistable reaction function $f$ be defined by \eqref{e:truncated:f}. Then:
\begin{enumerate}
    \item for each $a,K$ there exists $\underline{\lambda}>0$ such that there is no moving traveling wave of the form \eqref{e:wave}--\eqref{e:wave:core} for all $\lambda<\underline{\lambda}$ and all $\delta>0$,
    \item if $a \neq \frac{K}{2}$ then  there exists $\overline{\lambda}>0$ and $\overline{\delta}>0$ such that a moving traveling wave of the form \eqref{e:wave}--\eqref{e:wave:core} for all $\lambda>\overline{\lambda}$ and all $\delta>\overline{\delta}>0$.
\end{enumerate}
\end{theorem}

\begin{proof}
We only consider left traveling waves, the proof for right traveling waves follows analogously. To prove the first statement, we take $\lambda$ in $f$ \eqref{e:truncated:f} sufficiently small such that $\lambda<\overline{\lambda}$
\[
\overline{\lambda} p (K-p)(p-a)<1
\]  
for all $p\in(a,K)_\mathbb{Z}$. Consequently, \[f(p)\leq p+1<2p\] 
for all $p\in(a,K)_\mathbb{Z}$ and the condition \eqref{e:p:function} is violated. Thus, there is no left traveling wave.
    
    
To prove the second statement, consider $p=a+1$. Then we can determine a large enough $\overline{\lambda}$ such that the inequality \eqref{e:p:function} is satisfied for all $\lambda>\overline{\lambda}$, i.e.,
    \[
    \overline{\lambda}=\frac{1}{(K-p)(p-a)}=\frac{1}{K-a-1}.
    \]
Using $\overline{\lambda}$ in \eqref{e:truncated:f}, we get  \[f(a+1)=a+1+\frac{1}{K-a-1}(a+1)(K-a-1)=2a+2,\] i.e., $f(p)\geq 2p$ holds and the inequality \eqref{e:p:function} is satisfied for all $\lambda>\overline{\lambda}$. Consequently, for all $\delta\geq \overline{\delta}=a+1$, Thm.~\ref{t:ns:left} yields the existence of a left traveling wave.
\end{proof}

\begin{figure}
    \centering
    \subcaptionbox{$\lambda=0.1$}{\includegraphics[width=0.49\textwidth]{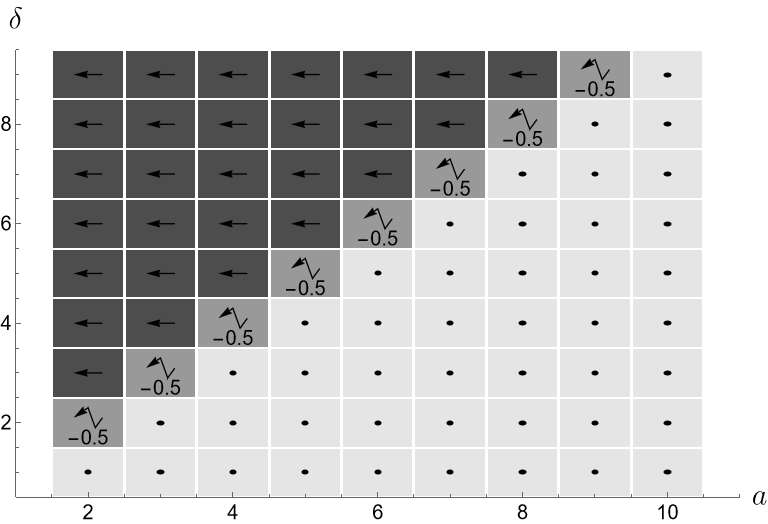}}    
    \subcaptionbox{$\lambda=0.05$}{\includegraphics[width=0.49\textwidth]{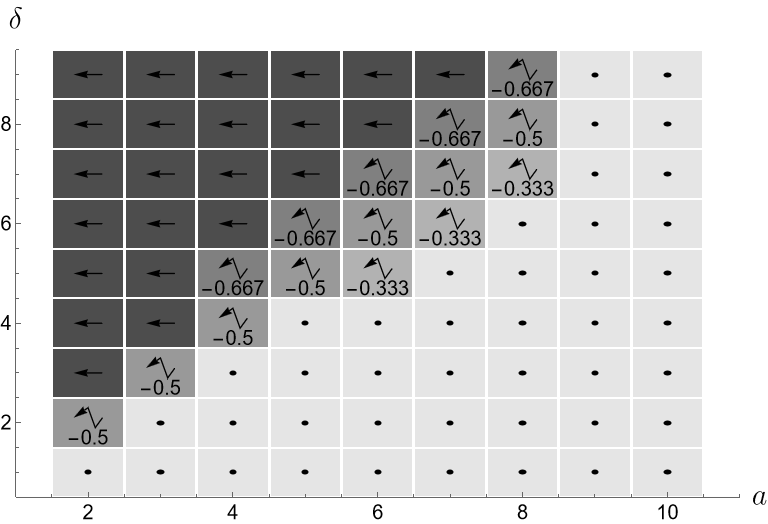}}
    \subcaptionbox{$\lambda=0.02$}{\includegraphics[width=0.49\textwidth]{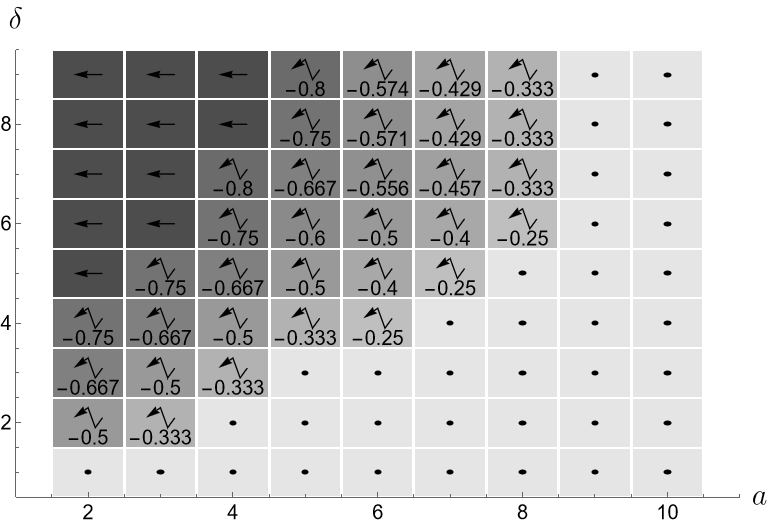}}
    \subcaptionbox{$\lambda=0.01$}{\includegraphics[width=0.49\textwidth]{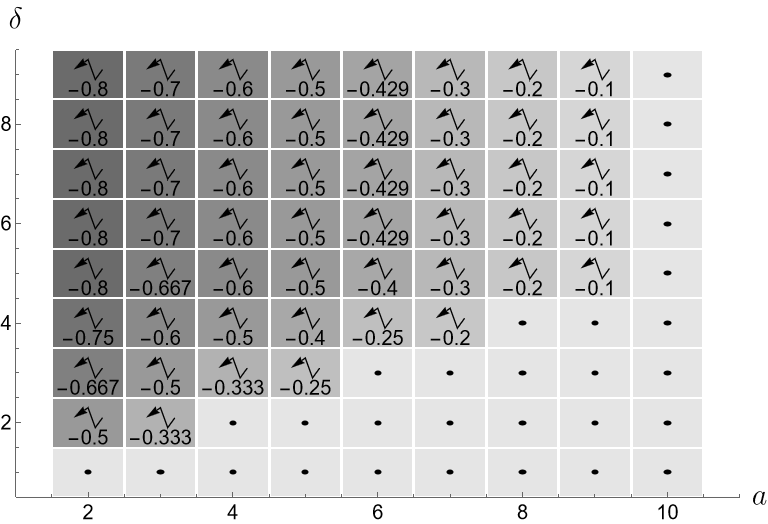}}
    \subcaptionbox{$\lambda=0.005$}{\includegraphics[width=0.49\textwidth]{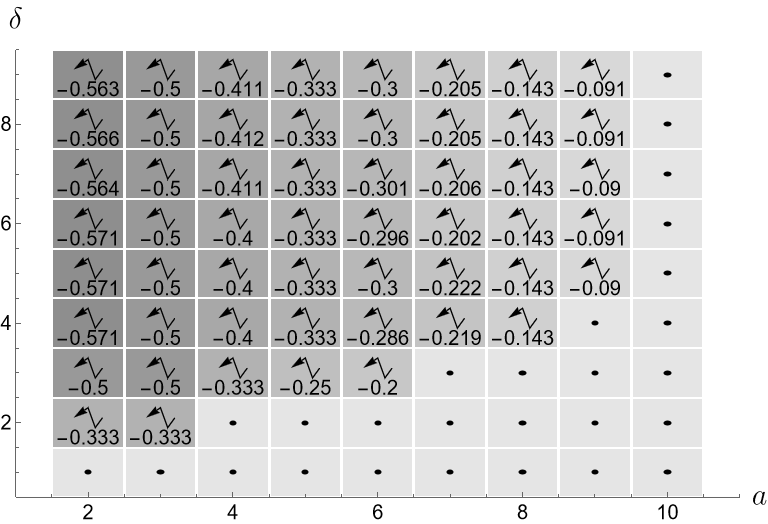}}
    \subcaptionbox{$\lambda=0.001$}{\includegraphics[width=0.49\textwidth]{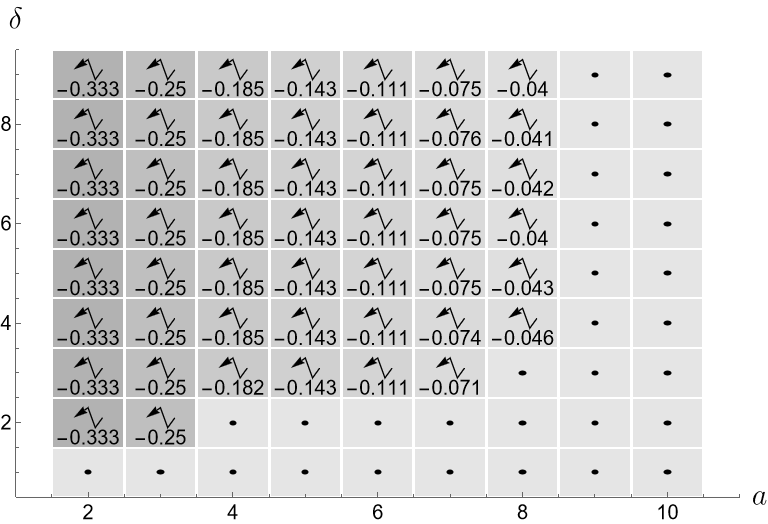}}
    \caption{Simulation results of existence of pinned (dot symbol), moving (straight arrow) and higher-order $(c,m)$-traveling waves (zigzag arrow) for $K=20$, various values of $a, \delta$, and the truncated polynomial bistable reaction function~\eqref{e:truncated:f} with 6 different values of $\lambda$ (each panel). 50 simulations have been performed for each of the combinations $a, \delta, \lambda$. The figure only depicts $a\leq K/2$, the values $a>K/2$ yields symmetric (right traveling) waves. The numerical values (and the gray shade) indicate the (average) speed $\gamma=c/m$ of higher-order $(c,m)$-waves~\eqref{e:gamma}.}
    \label{fig:x:polynomial}
\end{figure}

Let us highlight four key conclusions that follow from Thm.~\ref{t:truncated:f} and are illustrated in Fig.~\ref{fig:x:polynomial} and distinguish the behavior of the bistable RDCA~\eqref{e:RDCA:DDS} from the LDE~\eqref{e:NagumoLDE} (and  from the PDE~\eqref{e:NagumoPDE}).

Firstly, if the diffusion parameter $\delta$ is small, there are no moving traveling waves for any combination of parameters $a,K,\lambda$. For the LDE~\eqref{e:NagumoLDE} there always exist sufficiently small $a$ and $d$ for which there are moving traveling waves, see Fig.~\ref{fig:speed:comparison}.

Secondly, if $\lambda$ is large and $K$ is fixed, there exists a threshold value $\bar{\delta}$ such that there exists a moving traveling wave for all $\delta>\bar{\delta}$ and for all $a\neq \frac{K}{2}$. For the LDE~\eqref{e:NagumoLDE} there is no such finite value.

Thirdly, there exist (intermediate) values of $\lambda$ for which we cannot find any moving traveling waves for small values $\left\lvert a-\frac{K}{2}\right\rvert$ (for any $\delta$). For such $\lambda$, moving traveling waves exist only for sufficiently large $\left\lvert a-\frac{K}{2}\right\rvert$ and sufficiently large $\delta$.

Finally, if the reaction parameter $\lambda$ is small, there are no moving traveling waves for any combination of $a,K,\delta$. In this case, the pinned waves are accompanied by higher-order waves (cf. Fig.~\ref{fig:x:polynomial}(d)--(f)). The speed $\gamma$ of higher-order $(c,m)$-waves \eqref{e:gamma} decreases as $\lambda$ decreases.



Fig.~\ref{fig:x:polynomial} illustrates these differences and shows results of  simulations for various values of $a,\delta, \lambda$. Each panel corresponds to the reaction functions from Fig.~\ref{fig:x:polynomial:rf}. We performed 50 simulations for each combination of $a,\delta,\lambda$ with a random initial condition and analyzed the outcome after $100$  time iterations. The figure depicts pinned (dot symbol), moving (straight arrow) and higher-order $(c,m)$-traveling waves (zigzag arrow). In the latter case, the average speed $\gamma$ \eqref{e:gamma} is also depicted since the waves are not unique, as we demonstrated in previous sections. Note that the pinned waves symbol is also used for periodic pinned waves which occur rarely in certain parameter configurations as well.



\section*{Acknowledgements} Petr Stehl\'{\i}k gratefully acknowledges the support by the Czech Science Foundation grant no. GA22-18261S. The authors thank Vladim\'{\i}r \v{S}v\'{\i}gler and Jon\'{a}\v{s} Volek for helpful comments.

\bibliographystyle{abbrv}
\bibliography{ref}

\end{document}